\documentclass[pdflatex,sn-mathphys-num]{sn-jnl}
\usepackage{graphicx}%
\usepackage{multirow}%
\usepackage{amsmath,amssymb,amsfonts}%
\usepackage{amsthm}%
\usepackage{mathrsfs}%
\usepackage[title]{appendix}%
\usepackage{xcolor}%
\usepackage{textcomp}%
\usepackage{manyfoot}%
\usepackage{booktabs}%
\usepackage{algorithm}%
\usepackage{algorithmicx}%
\usepackage{algpseudocode}%
\usepackage{listings}%

\theoremstyle{thmstyleone}
\newtheorem{theorem}{Theorem}[section]
\newtheorem{lemma}[theorem]{Lemma}
\newtheorem{proposition}[theorem]{Proposition}

\theoremstyle{thmstylethree}
\newtheorem{definition}[theorem]{Definition}
\theoremstyle{thmstyletwo}
\newtheorem{remark}[theorem]{Remark}
\numberwithin{equation}{section}
\newcommand{\R}{\mathbb R}
\newcommand{\C}{\mathbb C}
\newcommand{\ip}[2]{\left(#1,#2\right)}
\newcommand{\norm}[1]{\left\lVert#1\right\rVert}
\allowdisplaybreaks[4]
\begin{document}

\title[Normalized solutions for a nonlinear Dirac equation]{Normalized solutions and small mass asymptotics for a nonlinear Dirac equation with a singular weight and competing powers}

\author[1]{\fnm{Hao} \sur{Wang}}\email{wanghao168@nenu.edu.cn}
\author*[2]{\fnm{Xiaojun} \sur{Chang}}\email{changxj100@nenu.edu.cn}

\affil[1]{\orgdiv{School of Mathematics and Statistics}, \orgname{Northeast Normal University}, \orgaddress{\city{Changchun}, \postcode{130024}, \state{Jilin}, \country{PR China}}}
\affil*[2]{\orgdiv{School of Mathematics and Statistics \& Center for Mathematics and Interdisciplinary Sciences}, \orgname{Northeast Normal University}, \orgaddress{\city{Changchun}, \postcode{130024}, \state{Jilin}, \country{PR China}}}

\abstract{We study normalized solutions of the nonlinear Dirac equation
\[
\begin{cases}
-i\alpha\cdot\nabla u+m\beta u=\mu u+|x|^{-b}\bigl(|u|^{q-2}-|u|^{p-2}\bigr)u&\text{in }\mathbb{R}^3,\\
\displaystyle\int_{\mathbb{R}^3}|u|^2\,dx=a,
\end{cases}
\]
where $m>0$, $0<b<1$, $2<q<3-b$, and $q<p<\infty$. For each fixed $p>q$ and all sufficiently small $a>0$, we prove the existence of a normalized solution $(u_a,\mu_a)$ with $u_a\in H^1(\mathbb{R}^3;\mathbb C^4)\cap L^\infty(\mathbb{R}^3;\mathbb C^4)$ and  $\mu_a\in(0,m)$. This covers the weighted Sobolev critical and supercritical ranges $p\ge3-b$. Our method combines a mass-preserving spectral reduction with a convex truncation and uniform $L^\infty$ estimates. For solutions arising from reduced minimization,  we determine the energy asymptotics and the rate at which $\mu_a$ approaches $m$ as $a\to0^+$. After rescaling, the upper two-component spinors converge in $H^{\frac12}(\mathbb{R}^3;\mathbb C^2)$ to the set of normalized ground states of a weighted nonlinear Schr\"odinger equation with exponent $q$, while the lower two-component spinors converge to zero in the same space. We also derive an asymptotic formula in $L^2(\mathbb{R}^3;\mathbb C^2)$ for the rescaled lower two-component spinor in terms of the gradient of the limiting Schr\"odinger ground state.}


\maketitle

\section{Introduction} In this paper, we deal with the normalized nonlinear Dirac problem
\begin{equation}\label{eq:main}
\begin{cases}
-i\alpha\cdot\nabla u+m\beta u=\mu u+|x|^{-b}\bigl(|u|^{q-2}-|u|^{p-2}\bigr)u&\text{in }\R^3,\\
\displaystyle\int_{\R^3}|u|^2\,dx=a,
\end{cases}
\end{equation}
where $u:\R^3\to\C^4$, $m>0$ is the mass of the Dirac particle, and $\alpha_1,\alpha_2,\alpha_3,\beta$ are the standard Dirac matrices. The $L^2$-mass $a>0$, also called the charge, is prescribed, while the frequency $\mu\in\R$ arises as a Lagrange multiplier. Throughout the paper, we assume that
\[
0<b<1,\quad 2<q<3-b,\quad q<p<\infty.
\]

The mass constraint has a natural interpretation in terms of standing
waves. Writing $H_0=-i\alpha\cdot\nabla+m\beta$, the associated evolution
equation is
\[
i\partial_t\psi=H_0\psi-|x|^{-b}
(|\psi|^{q-2}-|\psi|^{p-2})\psi.\]
It formally conserves the $L^2$ charge, and a standing wave
$\psi(t,x)=e^{-i\mu t}u(x)$ satisfies the stationary equation in
~\eqref{eq:main}. Prescribing the charge therefore leads to the second condition in~\eqref{eq:main}.

The nonlinear terms in \eqref{eq:main} represent attraction at the
lower power and repulsion at the higher power. For the Dirac
operator, this repulsion does not remove the strong indefiniteness
of the energy. Moreover, the weighted Sobolev threshold is $3-b$:
at this exponent compactness fails, and above it the energy need
not be finite on $H^{1/2}(\mathbb{R}^3;\mathbb{C}^4)$.
We seek normalized solutions for every fixed $p>q$ and all
sufficiently small masses, and determine how the two powers
contribute to their leading behavior as $a\to0^+$.

Variational methods for fixed-frequency nonlinear Dirac equations
were developed by Esteban and S\'er\'e~\cite{Esteban-1995};
see also~\cite{Esteban-1996,Bartsch-2006,Ding-2008,Ding-2012,Ding-2015,Figueiredo-2017,Wang-2021}.
For normalized solutions of nonlinear Dirac equations, the strong indefiniteness of the energy must be addressed simultaneously with the 
$L^2$-constraint.
An early constrained approach was developed by Buffoni and
Jeanjean~\cite{Buffoni-1993} for semilinear elliptic equations
with a spectral gap, using spectral reduction on $L^2$ spheres
and minimax arguments.
Buffoni, Esteban, and S\'er\'e~\cite{Buffoni-2006} subsequently
obtained normalized solutions for a class of strongly indefinite
semilinear elliptic equations by combining penalization,
variational Lyapunov--Schmidt reduction, and a mountain-pass
argument. Their approach drew on methods developed for the
Dirac--Fock equations in~\cite{Esteban-1999}.

Constrained reduction and minimax methods have since been used
for the Maxwell--Dirac system~\cite{Nolasco-2021}, the
Klein--Gordon--Dirac system~\cite{CotiZelati-2023}, and local
nonlinear Dirac equations~\cite{CotiZelati-2025}.
For spatially dependent nonlinearities, Ding, Yu, and
Zhao~\cite{Ding-2023} treated a power with a coefficient vanishing
at infinity, while Bahrouni et al.~\cite{Bahrouni-2023}
considered more general nonautonomous problems.

Competing powers have also been studied in normalized
Schr\"odinger equations. Soave~\cite{Soave-2020,Soave-2020-JFA}
considered a focusing higher power and a lower power of either
sign. With a defocusing lower power, Bellazzini, Forcella, and
Georgiev~\cite{BellazziniForcellaGeorgiev-2023} studied
ground-state energy thresholds and blow-up in three dimensions.
The reverse ordering, relevant to \eqref{eq:main}, occurs in the
cubic--quintic equation, for which Jeanjean and
Lu~\cite{JeanjeanLu-2022-M3AS} constructed positive energy
normalized solutions of local minimum and mountain-pass type
in suitable mass ranges.

For spatially inhomogeneous Schr\"odinger equations,
existence, stability, and dynamical properties of single-power
models, including fractional ones, have been studied
in~\cite{Genoud-2008,Farah-2016,Saanouni-2024}.
For combined weighted powers, Gou~\cite{Gou-2024} established
existence results for standing waves with prescribed $L^2$ norm.
Gou, Majdoub, and Saanouni~\cite{Gou-2026} investigated ground
states, blow-up, and scattering for competing inhomogeneous
nonlinearities with a focusing leading term and a defocusing
perturbation.

 For the Dirac equation, our previous
work~\cite{WangChang2026} established small mass existence and
weighted compactness for a single power nonlinear term of the form 
$|x|^{-b}|u|^{r-2}u$, where $0<b<1$ and
$2<r<3-b$.
The present paper treats a competing defocusing power with
arbitrary fixed exponent $p>q$, including $p\ge3-b$.
Our additional ingredient is a convex truncation together with
uniform amplitude estimates for the resulting solutions.
These estimates recover the original equation even when the
weighted $p$-power energy is not finite on the whole energy space.
Theorem~\ref{thm:small-mass} then identifies the leading energy,
frequency, and spinor asymptotics of the solutions selected by
reduced minimization.

The asymptotic results are related to nonrelativistic limits of
normalized Dirac equations.
Chen et al.~\cite{Chen-2024} obtained convergence to normalized
Schr\"odinger ground states in $H^1$ under regularity assumptions
on the spatial coefficients. In particular, their local coefficient
belongs to $C^1(\R^3)$, whereas the coefficient $|x|^{-b}$ in
\eqref{eq:main} is singular at the origin.
For an unweighted power nonlinearity, Chen, Ding, and
Guo~\cite{Chen-2025} studied energy and action ground states,
quantified the energy convergence, and obtained separate estimates
for the spectral components and the upper and lower two-component
spinors.  Further results for multiple normalized
bound states and for equations on metric graphs appear
in~\cite{ChenGuoYu-2026,Borrelli-2019,ChenGuoGraph2026}.

In the present problem, the Dirac operator is fixed and
$a\to0^+$.
After normalization and spatial rescaling, this becomes a
nonrelativistic limit in which the singular weight remains
and the coefficient of the higher-power term tends to zero.
Although the natural Dirac energy space has order $1/2$,
the positive spectral dispersion satisfies
\[
\sqrt{m^2+|\xi|^2}-m
=\frac{|\xi|^2}{\sqrt{m^2+|\xi|^2}+m}
\sim\frac{|\xi|^2}{2m}
\qquad\text{as }|\xi|\to0.
\]
The limiting problem therefore has the classical
Schr\"odinger kinetic energy.
We identify the precise energy and frequency coefficients
through its normalized ground states.
The profile convergence is proved in $H^{1/2}$, and an
additional amplitude rescaling of the lower two-component
spinor gives a strong $L^2$ limit expressed through the
gradient of a limiting ground state.
The proof combines sharp reduced energy estimates with
Fourier tail control and does not assume an $H^1$ bound
for the rescaled profiles.

We briefly describe the variational construction. The free operator has
\[
\sigma(H_0)=(-\infty,-m]\cup[m,\infty),\qquad
E=D(|H_0|^{\frac12})=H^{\frac12}(\R^3;\C^4)=E^+\oplus E^-.
\]
Its quadratic form is unbounded above and below on the $L^2$ sphere.
With
\[
F(\sigma)=\frac{\sigma^q}{q}-\frac{\sigma^p}{p},\qquad \sigma\ge0,
\]
the original energy is
\begin{equation}\label{eq:energy-introduction}
I(u)=\frac12\bigl(\|u^+\|_E^2-\|u^-\|_E^2\bigr)
-\int_{\R^3}|x|^{-b}F(|u|)\,dx,
\end{equation}
defined for $u\in E$ with
$\int_{\R^3}|x|^{-b}|u|^p\,dx<\infty$.
The primitive loses convexity at large amplitudes.
Moreover, the weighted embedding at $p=3-b$ is not compact,
and for $p>3-b$ the energy need not be finite on all of $E$.

To overcome these obstacles, we replace $F$ at large
amplitudes by a nonnegative convex $C^2$ function $F_\delta$
that agrees with $F$ near zero and has global $q$-growth.
Related truncation and penalization arguments appear
in~\cite{JeanjeanTanaka-2004,Alves-2025}.
The derivative bounds for $F_\delta$ and the smallness of
the mass yield uniform strict concavity on the negative
spectral fibers, a unique fiber maximizer, and a $C^1$
reduced functional.
The compact weighted embedding at exponent $q<3-b$ provides
the compactness required for reduced minimization.
An $L^r$ bootstrap based on $H_0^{-1}$ then gives
\[
\|u_a\|_{H^1}+\|u_a\|_\infty\le C\sqrt a,
\]
with $C$ independent of $a$.
Thus the truncation is inactive for sufficiently small mass,
and the constructed solutions solve \eqref{eq:main}.
Our first result is the following existence theorem.

\begin{theorem}\label{thm:main}
Let $m>0$, $0<b<1$, $2<q<3-b$, and $q<p<\infty$. Then there exists $a^*=a^*(m,b,q,p)>0$ such that, for every $a\in(0,a^*)$, problem~\eqref{eq:main} admits a normalized solution
\[
(u_a,\mu_a)\in
\bigl(H^1(\R^3;\C^4)\cap L^\infty(\R^3;\C^4)\bigr)\times(0,m).
\]
Moreover, there exists a constant $C=C(m,b,q,p)>0$, independent of $a$, such that, for every $a\in(0,a^*)$,
\[
\norm{u_a}_{H^1}+\norm{u_a}_\infty\le C\sqrt a,\quad 0<I(u_a)<\frac{ma}{2},
\]
where $I$ is the original, untruncated energy defined in~\eqref{eq:energy-introduction}.
\end{theorem}

We now state the small mass asymptotics for families selected by
reduced minimization. The scale is fixed by the balance between
the kinetic correction and the lower-power term. Set
\[
\gamma_q:=\frac{3(q-2)}2+b,\qquad
\varepsilon_a:=a^{\frac{q-2}{2(2-\gamma_q)}}
=a^{\frac{q-2}{10-2b-3q}}.
\]
Since $0<\gamma_q<2$, we have $\varepsilon_a\to0^+$ and
\[
a^{\frac q2}\varepsilon_a^{\gamma_q}=a\varepsilon_a^2.
\]
The limiting variational problem is defined on
\[
\mathcal S_q:=\{Q\in H^1(\R^3;\C^2):\|Q\|_2^2=1\}
\]
by the energy
\begin{equation}\label{eq:NLS-energy-introduction}
\mathscr E_q(Q)=\frac1{4m}\int_{\R^3}|\nabla Q|^2\,dx
-\frac1q\int_{\R^3}|x|^{-b}|Q|^q\,dx
\end{equation}
and the level and minimizer set
\begin{equation}\label{eq:NLS-ground-set-introduction}
e_q:=\inf_{Q\in\mathcal S_q}\mathscr E_q(Q),\qquad
\mathcal G_q:=\{Q\in\mathcal S_q:\mathscr E_q(Q)=e_q\}.
\end{equation}
We also set
\begin{equation}\label{eq:Lambda-limit-introduction}
\Lambda_q:=-\frac{2(q-\gamma_q)}{2-\gamma_q}e_q
\end{equation}
and
\[
\mathcal G_q^{\mathrm{up}}
:=\left\{\begin{pmatrix}Q\\0\end{pmatrix}:Q\in\mathcal G_q\right\}
\subset H^{\frac12}(\R^3;\C^4).
\]

\begin{theorem}\label{thm:small-mass} Assume the hypotheses of Theorem~\ref{thm:main}.
Fix an admissible truncation parameter $\delta$ as in
Section~\ref{sec:prelim}. With the notation introduced in
Sections~\ref{sec:reduction}--\ref{sec:truncated},
choose, for each sufficiently small $a>0$, any
$v_a\in X_a$ satisfying $J_\delta(v_a)=e_a^\delta$, and set
\[
u_a=G_\delta(v_a),\quad
\mu_a=\kappa_\delta(u_a).
\]
Keep $p$ and $\delta$ fixed as $a\to0^+$.
 Then $e_q<0$, $\Lambda_q>0$, and $\mathcal G_q$ is nonempty and compact in $H^1(\R^3;\C^2)$. Every $Q\in\mathcal G_q$ is a weak solution of
\begin{equation}\label{eq:NLS-limit-introduction}
\begin{cases}
-\dfrac1{2m}\Delta Q+\Lambda_q Q=|x|^{-b}|Q|^{q-2}Q&\text{in }\R^3,\\
\displaystyle\int_{\R^3}|Q|^2\,dx=1.
\end{cases}
\end{equation}
Every family obtained in this way satisfies
\begin{equation}\label{eq:main-energy-asymptotic}
I(u_a)=\frac{ma}{2}+a\varepsilon_a^2e_q+o(a\varepsilon_a^2),\quad
\frac{m-\mu_a}{\varepsilon_a^2}\rightarrow\Lambda_q\quad\text{as }a\to0.
\end{equation}
Define
\begin{equation}\label{eq:main-rescaling}
\widetilde u_a(x):=a^{-\frac12}\varepsilon_a^{-\frac32} u_a\left(\frac{x}{\varepsilon_a}\right).
\end{equation}
Then
\begin{equation}\label{eq:main-profile-convergence}
\operatorname{dist}_{H^{\frac12}(\R^3;\C^4)}\bigl(\widetilde u_a,\mathcal G_q^{\mathrm{up}}\bigr)\rightarrow0
\quad\text{as }a\to0.
\end{equation}
Every sequence $a_n\to0^+$ admits a subsequence, still denoted by $a_n$, and a function $Q\in\mathcal G_q$ such that
\[
\widetilde u_{a_n}\rightarrow
\begin{pmatrix}
Q\\
0
\end{pmatrix}
\quad\text{strongly in }H^{\frac12}(\R^3;\C^4).
\]
Along the same subsequence, write $u_{a_n}$ in the standard $2+2$ form as
\[
u_{a_n}=
\begin{pmatrix}
\varphi_{a_n}\\
\chi_{a_n}
\end{pmatrix},
\quad
\varphi_{a_n},\chi_{a_n}\in H^1(\R^3;\C^2).
\]
Then the lower two-component spinor satisfies
\begin{equation}\label{eq:main-lower-spinor-limit}
a_n^{-\frac12}\varepsilon_{a_n}^{-\frac52}\chi_{a_n}\left(\frac{\,\cdot\,}{\varepsilon_{a_n}}\right)
\rightarrow-\frac{i}{2m}\boldsymbol{\sigma}\cdot\nabla Q\quad\text{strongly in }L^2(\R^3;\C^2),
\end{equation}
where $\boldsymbol{\sigma}=(\sigma_1,\sigma_2,\sigma_3)$ is the vector of Pauli matrices and
$\boldsymbol{\sigma}\cdot\nabla Q:=\sum\limits_{j=1}^3\sigma_j\partial_jQ$. The unscaled solutions satisfy
\begin{equation}\label{eq:main-unscaled-limit}
\norm{u_a}_{H^1}+\norm{u_a}_\infty\rightarrow0,\quad\frac{u_a}{\sqrt a}\rightharpoonup0\quad\text{in }H^1(\R^3;\C^4).
\end{equation}
\end{theorem}

\begin{remark}\label{rem:two-scale-asymptotics}
For $r\in\{q,p\}$, write
\[
\mathcal A_r(u)
:=
\int_{\R^3}|x|^{-b}|u|^r\,dx.
\]
For any family selected as in Theorem~\ref{thm:small-mass},
we also have, as $a\to0^+$,
\[
\frac{\mathcal A_q(u_a)}{a\varepsilon_a^2}
\rightarrow
-\frac{2q}{2-\gamma_q}e_q,
\quad
\frac{\mathcal A_p(u_a)}{a\varepsilon_a^2}
\rightarrow0.
\]
Thus, for every fixed $p>q$, the higher-power term is
negligible on the energy scale $a\varepsilon_a^2$.
The first energy correction and the asymptotics of $m-\mu_a$
are determined by the limiting weighted Schr\"odinger problem
with exponent $q$.

The negative spectral component $u_a^-$ should be
distinguished from the lower two-component spinor.
Writing $u_a=(\varphi_a,\chi_a)^{\mathsf T}$ in the standard
$2+2$ representation, we have
\[
\frac{\|u_a^-\|_E}{\sqrt a\,\varepsilon_a}
\rightarrow0,
\quad
\frac{\|\chi_a\|_2^2}{a\varepsilon_a^2}
\rightarrow
-\frac{\gamma_q}{m(2-\gamma_q)}e_q>0.
\]
\end{remark}

To make the connection with a nonrelativistic limit explicit,
consider the rescaled spinor $\widetilde u_a$
in~\eqref{eq:main-rescaling} and set
\[
c_a=\varepsilon_a^{-1},
\quad
\eta_a=
\bigl(\sqrt a\,\varepsilon_a^{\frac32}\bigr)^{p-q}.
\]
Using the homogeneity of $|x|^{-b}$ and the balance relation
\[
a^{\frac{q-2}{2}}\varepsilon_a^{\gamma_q}
=
\varepsilon_a^2,
\]
we can rewrite \eqref{eq:main} as
\[
\bigl(-ic_a\alpha\cdot\nabla+mc_a^2\beta\bigr)\widetilde u_a={}\frac{\mu_a}{\varepsilon_a^2}\widetilde u_a+|x|^{-b}
\bigl(|\widetilde u_a|^{q-2}-\eta_a|\widetilde u_a|^{p-2}\bigr)\widetilde u_a,
\]
with $\|\widetilde u_a\|_2^2=1$.
As $a\to0^+$, we have $c_a\to\infty$ and $\eta_a\to0^+$,
and, by \eqref{eq:main-energy-asymptotic},
\[
mc_a^2-\frac{\mu_a}{\varepsilon_a^2}=\frac{m-\mu_a}{\varepsilon_a^2}\rightarrow\Lambda_q.
\]
Thus the small mass limit can be interpreted as a
nonrelativistic limit at unit $L^2$ mass, in which the
coefficient of the $p$-term tends to zero while the
singular weight $|x|^{-b}$ remains unchanged.

The paper is organized as follows. Section~\ref{sec:prelim} introduces the functional setting, the weighted embeddings, and the explicit truncation. Section~\ref{sec:reduction} establishes the constrained spectral reduction. Section~\ref{sec:truncated} constructs a normalized solution of the truncated equation. Section~\ref{sec:recovery} establishes uniform $L^\infty(\R^3;\C^4)$ and $H^1(\R^3;\C^4)$ estimates, recovers
the original equation, and completes the proof of Theorem~\ref{thm:main}. Section~\ref{sec:small-mass} studies the rescaled minimizers and derives the energy,
frequency, and lower-spinor asymptotics in
Theorem~\ref{thm:small-mass}.

\section{Functional setting and convex truncation}\label{sec:prelim}

Throughout the rest of the paper, the parameters $m$, $b$, $q$, and $p$ are fixed and satisfy the assumptions of Theorem~\ref{thm:main}. The truncation parameter $\delta$ will be chosen below, depending only on $p$ and $q$, and then kept fixed. All smallness conditions on $a$ are imposed with these parameters fixed.  The symbol $C>0$ denotes a generic constant that may change from line to line. Unless otherwise stated, it is independent of $a$ but may depend on the fixed parameters and on other quantities already fixed in the argument.

For $1\le r\le\infty$, we write $\norm{u}_r=\norm{u}_{L^r(\R^3;\C^N)}$ and use $\norm{u}_{H^s}$ and $\norm{u}_{W^{1,r}}$ for the norms in $H^s(\R^3;\C^N)$ and $W^{1,r}(\R^3;\C^N)$, respectively, where $N\in\{2,4\}$ is determined by $u$. The same notation is used for real-valued functions. Here $\|\cdot\|_{\mathrm{op}}$ denotes the operator norm induced by the Euclidean norm.

We use the unitary Fourier transform
\[
\widehat z(\xi)=(2\pi)^{-3/2}
\int_{\R^3}e^{-ix\cdot\xi}z(x)\,dx
\]
initially for integrable functions and then by its usual
extension to $L^2$ and tempered distributions. Thus
$\widehat{-i\partial_j z}(\xi)=\xi_j\widehat z(\xi)$.
Our Sobolev norms are
\[
\|z\|_{H^s}^2=\int_{\R^3}(1+|\xi|^2)^s
|\widehat z(\xi)|^2\,d\xi,
\qquad
\|z\|_{\dot H^s}^2=\int_{\R^3}|\xi|^{2s}
|\widehat z(\xi)|^2\,d\xi
\]
whenever the corresponding integrals are finite.

\subsection{The Dirac operator and weighted embeddings}
We use the standard $2+2$ block representation of the Dirac matrices
\[
\alpha_k=
\begin{pmatrix}
0&\sigma_k\\
\sigma_k&0
\end{pmatrix},
\quad
\beta=
\begin{pmatrix}
I_2&0\\
0&-I_2
\end{pmatrix},
\quad k=1,2,3,
\]
where $I_2$ is the $2\times2$ identity matrix and
\[
\sigma_1=
\begin{pmatrix}
0&1\\
1&0
\end{pmatrix},
\quad
\sigma_2=
\begin{pmatrix}
0&-i\\
i&0
\end{pmatrix},
\quad
\sigma_3=
\begin{pmatrix}
1&0\\
0&-1
\end{pmatrix}
\]
are the Pauli matrices. We also write
\[
\boldsymbol{\sigma}=(\sigma_1,\sigma_2,\sigma_3),
\quad
\boldsymbol{\sigma}\cdot\xi
=\sum_{k=1}^3\sigma_k\xi_k,
\quad \xi\in\R^3.
\]

For $N\in\{2,4\}$, we regard $L^2(\R^3;\C^N)$ as a real Hilbert space with inner product
\[
\ip{u}{v}_2:=\Re\int_{\R^3}u\cdot\overline v\,dx.
\]
The free Dirac operator $H_0=-i\alpha\cdot\nabla+m\beta$ is self-adjoint on $L^2(\R^3;\C^4)$ with $D(H_0)=H^1(\R^3;\C^4)$ and
\[
\sigma(H_0)=(-\infty,-m]\cup[m,+\infty).
\]
Let $P^\pm$ denote the spectral projections associated with the positive and negative parts of $\sigma(H_0)$, respectively, and set
\[
L^\pm=P^\pm L^2(\R^3;\C^4).
\]
Then $L^2(\R^3;\C^4)=L^+\oplus L^-$ is an orthogonal decomposition.

The natural energy space associated with $H_0$ is $E:=D\bigl(|H_0|^{\frac12}\bigr)$, endowed with
\[
\ip{u}{v}_E=\ip{|H_0|^{\frac12}u}{|H_0|^{\frac12}v}_2,\quad\norm{u}_E^2=\ip{u}{u}_E.
\]
Since $H_0^2=(-\Delta+m^2)I_4$, we have $E=D\bigl(|H_0|^{\frac12}\bigr)\cong H^{\frac12}(\R^3;\C^4)$ with equivalent norms. More precisely,
\begin{equation}\label{eq:normequiv}
m\norm{u}_2^2\le\norm{u}_E^2,\quad\min\{m,1\}\norm{u}_{H^{\frac12}}^2\le\norm{u}_E^2
\le\max\{m,1\}\norm{u}_{H^{\frac12}}^2,\quad u\in E.
\end{equation}
Setting $E^\pm=P^\pm E=E\cap L^\pm$, we obtain the $E$-orthogonal decomposition $E=E^+\oplus E^-$. Hence every $u\in E$ admits a unique decomposition $u=u^++u^-$ with $u^\pm=P^\pm u$. The quadratic form associated with $H_0$ extends continuously to $E$ and satisfies
\[
\langle H_0u,v\rangle_{E^*,E}=\ip{u^+}{v^+}_E-\ip{u^-}{v^-}_E,\quad\langle H_0u,u\rangle_{E^*,E}=\norm{u^+}_E^2-\norm{u^-}_E^2.
\]

In Fourier variables, $H_0$ has the symbol $\widehat H_0(\xi)=\alpha\cdot\xi+m\beta$. To diagonalize $\widehat H_0(\xi)$, set
\[
\lambda(\xi)=\sqrt{m^2+|\xi|^2},\quad\rho_\pm(\xi)=\left(\frac{\lambda(\xi)\pm m}{2\lambda(\xi)}\right)^{\frac12}
\]
and
\[
U(\xi)=\frac{\bigl(m+\lambda(\xi)\bigr)I_4+\beta\alpha\cdot\xi}{\sqrt{2\lambda(\xi)\bigl(m+\lambda(\xi)\bigr)}}.
\]
Then $U(\xi)$ is unitary and
\[
U(\xi)\widehat H_0(\xi)U(\xi)^{-1}=\lambda(\xi)\beta.
\]
Consequently, the Foldy--Wouthuysen transformation $U_{\rm FW}=\mathcal F^{-1}U\mathcal F$ is unitary on $L^2(\R^3;\C^4)$ and satisfies
\[
U_{\rm FW}H_0U_{\rm FW}^{-1}=\beta|H_0|.
\]
These standard facts about the free Dirac operator and its Foldy--Wouthuysen diagonalization may be found, for example, in~\cite[Chapter~1]{Thaller1992}.

We use the notation $\mathcal A_r$ introduced in Remark~\ref{rem:two-scale-asymptotics} for every $r>2$ and measurable $u:\R^3\to\C^N$, where $N\in\{2,4\}$, allowing $\mathcal A_r(u)=+\infty$. For $1\le r<\infty$ and $u\in L^r(\R^3,|x|^{-b}\,dx;\C^N)$, we write
\[
\norm{u}_{L_b^r}:=\left(\int_{\R^3}|x|^{-b}|u|^r\,dx\right)^{\frac1r}.
\]

The subcritical compact embedding for $\C^4$-valued functions
was proved in \cite[Lemma~2.1]{WangChang2026}.
We include a proof that also covers $\C^2$-valued functions
and establishes continuity of the embedding at the endpoint
$r=3-b$.

\begin{lemma}\label{lem:weighted}
Let $0<b<1$ and $N\in\{2,4\}$. For every $2<r\le3-b$, the embedding
\[
H^{\frac12}(\R^3;\C^N)\hookrightarrow L^r(\R^3,|x|^{-b}\,dx;\C^N)
\]
is continuous and is compact whenever $2<r<3-b$, but it is not compact at $r=3-b$. Moreover,
\[
\mathcal A_r(u)\le C_r\norm{u}_{H^{\frac12}}^r,\quad u\in H^{\frac12}(\R^3;\C^N).
\]
In particular, for $N=4$ and $u\in E$,
\begin{equation}\label{eq:weighted}
\mathcal A_r(u)\le C_r\norm{u}_E^r.
\end{equation}
For every $r>3-b$, there exists $u\in H^{\frac12}(\R^3;\C^N)$ such that $\mathcal A_r(u)=\infty$.
\end{lemma}

\begin{proof}
At the endpoint $r=3-b$, applying fractional Hardy inequality~\cite{BahouriCheminDanchin2011} to each component of $u$ and summing, we obtain
\[
\int_{\R^3}\frac{|u|^2}{|x|}\,dx\le C\norm{u}_{\dot H^{\frac12}}^2.
\]
Combining this estimate with critical Sobolev embedding $\dot H^{\frac12}(\R^3;\C^N)\hookrightarrow L^3(\R^3;\C^N)$ and H\"{o}lder's inequality, we have
\[
\mathcal A_{3-b}(u)\le\left(\int_{\R^3}\frac{|u|^2}{|x|}\,dx\right)^b\left(\int_{\R^3}|u|^3\,dx\right)^{1-b}
\le C\norm{u}_{\dot H^{\frac12}}^{3-b}\le C\norm{u}_{H^{\frac12}}^{3-b}.
\]
For $2<r<3-b$, choose $\theta>1$ such that
\[
\frac3{3-r}<\theta<\frac3b
\]
and let $\theta'=\frac{\theta}{\theta-1}$. Then $b\theta<3$ and $r\theta'<3$, which implies
\[
\int_{B_1}|x|^{-b}|u|^r\,dx\le\left(\int_{B_1}|x|^{-b\theta}\,dx\right)^{\frac1\theta}\norm{u}_{r\theta'}^r
\le C\norm{u}_{H^{\frac12}}^r.
\]
Since $|x|^{-b}\le1$ on $\R^3\setminus B_1$, we also have
\[
\int_{\R^3\setminus B_1}|x|^{-b}|u|^r\,dx\le\norm{u}_r^r\le C\norm{u}_{H^{\frac12}}^r.
\]
For $u\in E$, by \eqref{eq:normequiv}, we obtain
\[
\mathcal A_r(u)
\le C_r\norm{u}_{H^{\frac12}}^r
\le C_r\norm{u}_E^r,
\quad 2<r\le3-b,
\]
which proves \eqref{eq:weighted}.

We now prove that the embedding $H^{\frac12}(\R^3;\C^N)\hookrightarrow L^r(\R^3,|x|^{-b}\,dx;\C^N)$ is compact for $2<r<3-b$. Let $\{u_n\}$ be a bounded sequence in $H^{\frac12}(\R^3;\C^N)$. Passing to a subsequence, we may assume that $u_n\rightharpoonup u$ in $H^{\frac12}(\R^3;\C^N)$. Set $z_n=u_n-u$. For $R>1$,
\[
\int_{\R^3\setminus B_R}|x|^{-b}|z_n|^r\,dx\le R^{-b}\norm{z_n}_r^r\le C R^{-b}.
\]
Fix $0<\varepsilon<1$. Since the embedding $H^{\frac12}(B_R;\C^N)\hookrightarrow L^r(B_R;\C^N)$ is compact for $r<3$, we deduce that
\[
z_n\to0\quad\text{strongly in }L^r(B_R;\C^N).
\]
Moreover, $|x|^{-b}\le\varepsilon^{-b}$ on $B_R\setminus B_\varepsilon$. Hence
\[
\int_{B_R\setminus B_\varepsilon}|x|^{-b}|z_n|^r\,dx\le\varepsilon^{-b}\norm{z_n}_{L^r(B_R)}^r\to0.
\]
By H\"{o}lder's inequality, we have
\[
\int_{B_\varepsilon}|x|^{-b}|z_n|^r\,dx\le\left(\int_{B_\varepsilon}|x|^{-b\theta}\,dx\right)^{\frac1\theta}
\norm{z_n}_{r\theta'}^r\le C\varepsilon^{\frac3\theta-b}.
\]
Combining the estimates on $\R^3\setminus B_R$,
$B_R\setminus B_\varepsilon$, and $B_\varepsilon$,
we obtain
\[
\limsup_{n\to\infty}
\int_{\R^3}|x|^{-b}|z_n|^r\,dx
\le C\left(R^{-b}
+\varepsilon^{\frac3\theta-b}\right).
\]
Letting $R\to\infty$ and $\varepsilon\to0^+$, we conclude that
\[
u_n\to u\quad\text{in }
L^r(\R^3,|x|^{-b}\,dx;\C^N).
\]

To show that the embedding is not compact for $r=3-b$,
let $\eta\in C_c^\infty(\R^3;\C^N)\setminus\{0\}$ and set
$\eta_\lambda(x)=\lambda\eta(\lambda x)$. Then
\[
\norm{\eta_\lambda}_2^2=\lambda^{-1}\norm{\eta}_2^2,\quad
\norm{\eta_\lambda}_{\dot H^{\frac12}}^2=\norm{\eta}_{\dot H^{\frac12}}^2,
\]
which implies that $\{\eta_\lambda\}_{\lambda\ge1}$ is bounded in $H^{\frac12}(\R^3;\C^N)$. Since $\norm{\eta_\lambda}_2\to0^+$ as $\lambda\to\infty$, we have $\eta_\lambda\rightharpoonup0$ in $H^{\frac12}(\R^3;\C^N)$. However, $\mathcal A_{3-b}(\eta_\lambda)=\mathcal A_{3-b}(\eta)>0$. Thus the sequence $\{\eta_n\}_{n\ge1}$ has no strongly convergent subsequence in
$L^{3-b}(\R^3,|x|^{-b}\,dx;\C^N)$.

Let $r>3-b$ and put $d=r+b-3>0$. Choose a nonzero function $\zeta\in C_c^\infty(\R^3;\C^N)$  with $\operatorname{supp}\zeta\subset\{x\in\R^3:1<|x|<2\}$. For $\lambda>1$,  set $\zeta_\lambda(x)=\lambda\zeta(\lambda x)$.
Then
\[
\norm{\zeta_\lambda}_{\dot H^{\frac12}}=\norm{\zeta}_{\dot H^{\frac12}},\quad
\norm{\zeta_\lambda}_2^2=\lambda^{-1}\norm{\zeta}_2^2,\quad
\mathcal A_r(\zeta_\lambda)=\lambda^d\mathcal A_r(\zeta).
\]
Fix an integer $L>\max\{2,\frac{r}{d}\}$ and define $\lambda_n=2^{Ln}$ and
\[
u=\sum_{n=1}^\infty 2^{-n}\zeta_{\lambda_n}.
\]
Since $\norm{\zeta_{\lambda_n}}_{H^{\frac12}}\le C$, we have
\[
\sum_{n=1}^\infty 2^{-n}
\norm{\zeta_{\lambda_n}}_{H^{\frac12}}
\le C\sum_{n=1}^\infty 2^{-n}<\infty.
\]
Thus $u\in H^{\frac12}(\R^3;\C^N)$.
Moreover, since $L>2$,
\[
\operatorname{supp}\zeta_{\lambda_n}
\subset\{x\in\R^3:2^{-Ln}<|x|<2^{1-Ln}\},
\quad
\operatorname{supp}\zeta_{\lambda_n}
\cap\operatorname{supp}\zeta_{\lambda_k}
=\varnothing\quad(n\ne k).
\]
Consequently,
\[
\mathcal A_r(u)
=\mathcal A_r(\zeta)\sum_{n=1}^\infty
2^{-nr}\lambda_n^d
=\mathcal A_r(\zeta)\sum_{n=1}^\infty
2^{n(Ld-r)}=\infty.
\]
\end{proof}

\begin{lemma}\label{lem:Ar-C2}
For every $2<r\le3-b$, the functional $\mathcal A_r:E\to\R$ is of class $C^2$, with
\[
\mathcal A_r'(u)[\phi]=r\Re\int_{\R^3}|x|^{-b}|u|^{r-2}u\cdot\overline\phi\,dx,
\]
\[
\mathcal A_r''(u)[\phi,\psi]=r\int_{\R^3}|x|^{-b}|u|^{r-2}\Re(\phi\cdot\overline\psi)\,dx
+r(r-2)\int_{\R^3}|x|^{-b}|u|^{r-4}\Re(u\cdot\overline\phi)\Re(u\cdot\overline\psi)\,dx,
\]
where $\C^4$ is viewed as a real Euclidean space. The product $|u|^{r-4}\Re(u\cdot\overline\phi)\Re(u\cdot\overline\psi)$ is defined to be zero wherever $u=0$. Moreover,
\[
\left|\mathcal A_r'(u)[\phi]\right|\le C_r\norm{u}_E^{r-1}\norm{\phi}_E,\quad
\left|\mathcal A_r''(u)[\phi,\psi]\right|\le C_r\norm{u}_E^{r-2}\norm{\phi}_E\norm{\psi}_E.
\]
\end{lemma}

\begin{proof}
By H\"older's inequality and Lemma~\ref{lem:weighted}, we have
\[
\begin{aligned}
|\mathcal A_r'(u)[\varphi]|
\le r\int_{\R^3}|x|^{-b}|u|^{r-1}|\varphi|\,dx
\le r\norm{u}_{L_b^r}^{r-1}\norm{\varphi}_{L_b^r}
\le C_r\norm{u}_E^{r-1}\norm{\varphi}_E.
\end{aligned}
\]
For the second derivative, using
\[
\left|\Re(\phi\cdot\overline\psi)\right|\le|\phi||\psi|,\quad
\left|\Re(u\cdot\overline\phi)\Re(u\cdot\overline\psi)\right|\le|u|^2|\phi||\psi|,
\]
we obtain
\[
\begin{aligned}
|\mathcal A_r''(u)[\phi,\psi]|&\le C_r\int_{\R^3}|x|^{-b}|u|^{r-2}|\phi||\psi|\,dx\\
&\le C_r\norm{u}_{L_b^r}^{r-2}\norm{\phi}_{L_b^r}\norm{\psi}_{L_b^r}\\
&\le C_r\norm{u}_E^{r-2}\norm{\phi}_E\norm{\psi}_E.
\end{aligned}
\]
To prove the continuity of $\mathcal A_r''$, let $u_n\to u$ in $E$. By Lemma~\ref{lem:weighted}, we also have $u_n\to u$ in $L^r(\R^3,|x|^{-b}\,dx;\C^4)$. Since $2<r\le3-b<3$, the maps
\[
T_1(z)=|z|^{r-2}I,\quad
T_2(z)=
\begin{cases}
|z|^{r-4}z\otimes z,&z\ne0,\\
0,&z=0,
\end{cases}
\]
are $(r-2)$-H\"{o}lder continuous on $\C^4$ viewed as a real Euclidean space. Here $(z\otimes z)h=\Re(z\cdot\overline h)z$. Consequently,
\[
\begin{aligned}
\norm{T_j(u_n)-T_j(u)}_{L_b^{\frac r{r-2}}}
&=\left(
\int_{\R^3}|x|^{-b}
\norm{T_j(u_n(x))-T_j(u(x))}_{\mathrm{op}}^{\frac r{r-2}}
\,dx
\right)^{\frac{r-2}{r}}\\
&\le C\norm{u_n-u}_{L_b^r}^{r-2}\to0,
\quad j=1,2.
\end{aligned}
\]
By H\"{o}lder's inequality and Lemma~\ref{lem:weighted}, we obtain
\[
\norm{\mathcal A_r''(u_n)-\mathcal A_r''(u)}_{\mathscr L(E,E^*)}\le C\sum_{j=1}^2\norm{T_j(u_n)-T_j(u)}_{L_b^{\frac{r}{r-2}}}\to0.
\]
Thus $u\mapsto\mathcal A_r''(u)$ is continuous from $E$ into $\mathscr L(E,E^*)$, and hence $\mathcal A_r\in C^2(E,\R)$.
\end{proof}

\begin{definition}\label{def:original-energy-domain}
For the fixed exponent $p>q$, let
\[
\mathcal D_p=\left\{u\in E:\mathcal A_p(u)<\infty\right\}.
\]
The original energy is defined on $\mathcal D_p$ by
\begin{equation}\label{eq:original-energy-domain}
I(u)=\frac12\norm{u^+}_E^2-\frac12\norm{u^-}_E^2-\frac1q\mathcal A_q(u)+\frac1p\mathcal A_p(u).
\end{equation}
If $p\le3-b$, by Lemma~\ref{lem:Ar-C2}, we have $\mathcal D_p=E$ and $I\in C^2(E,\R)$. If $p>3-b$, it follows from Lemma~\ref{lem:weighted} that
$\mathcal D_p\subsetneq E$, so $I$ is not a finite-valued functional on the whole space $E$.
\end{definition}

A normalized weak solution of \eqref{eq:main} is a pair
$(u,\mu)\in\mathcal D_p\times\R$ with $\|u\|_2^2=a$ such that,
for every $\phi\in C_c^\infty(\R^3;\C^4)$,
\[
\begin{aligned}
\langle H_0u,\phi\rangle_{E^*,E}
&=\mu\ip{u}{\phi}_2\\
&\quad+\Re\int_{\R^3}|x|^{-b}
\bigl(|u|^{q-2}-|u|^{p-2}\bigr)
u\cdot\overline\phi\,dx.
\end{aligned}
\]
The nonlinear term is locally integrable for $u\in\mathcal D_p$,
by H\"older's inequality with respect to $|x|^{-b}\,dx$.
The energy of such a solution is understood as in
\eqref{eq:original-energy-domain}.
For the solutions constructed below, the equation in fact holds
in $L^2(\R^3;\C^4)$, by their $H^1\cap L^\infty$ regularity and
Lemma~\ref{lem:original-regular-branch}.

\begin{lemma}\label{lem:original-regular-branch}
If $u\in H^1(\R^3;\C^4)\cap L^\infty(\R^3;\C^4)$, then $u\in\mathcal D_p$ and
\begin{equation}\label{eq:Ap-by-Aq}
\mathcal A_p(u)\le\norm{u}_\infty^{p-q}\mathcal A_q(u).
\end{equation}
Moreover,
\begin{equation}\label{eq:original-source-L2}
|x|^{-b}\bigl(|u|^{q-2}+|u|^{p-2}\bigr)u\in L^2(\R^3;\C^4).
\end{equation}
\end{lemma}

\begin{proof}
Since $u\in H^1(\R^3;\C^4)\subset H^{\frac12}(\R^3;\C^4)$
and $q<3-b$, we have $\mathcal A_q(u)<\infty$
by Lemma~\ref{lem:weighted}.
Using $|u|^p\le\norm{u}_\infty^{p-q}|u|^q$, we obtain
\eqref{eq:Ap-by-Aq}, and hence $u\in\mathcal D_p$.

To prove \eqref{eq:original-source-L2}, we estimate the nonlinear term
separately on $B_1$ and on $\R^3\setminus B_1$.
Since $b<1$ and $u\in L^\infty(\R^3;\C^4)$, we have
\[
\norm{|x|^{-b}|u|^{q-1}}_{L^2(B_1)}
\le
\norm{|x|^{-b}}_{L^2(B_1)}
\norm{u}_\infty^{q-1}
<\infty.
\]
On $\R^3\setminus B_1$, using $|x|^{-b}\le1$, we deduce that
\[
\norm{|x|^{-b}|u|^{q-1}}_{L^2(\R^3\setminus B_1)}
\le
\norm{u}_\infty^{q-2}\norm{u}_2
<\infty.
\]
Moreover,
\[
|x|^{-b}|u|^{p-1}
\le
\norm{u}_\infty^{p-q}|x|^{-b}|u|^{q-1},
\]
and hence~\eqref{eq:original-source-L2} follows.
\end{proof}

\subsection{The truncated primitive}
We construct a convex truncation of $F$ with global $q$-growth.
Fix $\delta\in(0,1)$, depending only on $p$ and $q$, such that
\begin{equation}\label{eq:delta-choice}
\delta^{p-q}<\frac{q-2}{p-2}.
\end{equation}
Set
\begin{equation}\label{eq:trunc-coefficients}
A_\delta=\frac1q-\frac{p-2}{q(q-2)}\delta^{p-q},
\quad
B_\delta=\frac{p-q}{2(q-2)}\delta^{p-2},
\quad
\Gamma_\delta=-\frac{(p-q)(p-2)}{2pq}\delta^p,
\end{equation}
and define
\begin{equation}\label{eq:trunc-definition}
F_\delta(\sigma)=
\begin{cases}
\displaystyle
\frac{\sigma^q}{q}-\frac{\sigma^p}{p},
&0\le\sigma\le\delta,\\
A_\delta\sigma^q+B_\delta\sigma^2+\Gamma_\delta,
&\sigma\ge\delta.
\end{cases}
\end{equation}
By \eqref{eq:delta-choice} and $p>q>2$, we have
$A_\delta>0$, $B_\delta>0$, and $\Gamma_\delta<0$.
\begin{lemma}\label{lem:truncation}
The function $F_\delta$ defined in \eqref{eq:trunc-definition} belongs to $C^2([0,\infty))$. Moreover, there are constants $c_\delta,C_\delta>0$ such that, for every $\sigma\ge0$,
\begin{equation}\label{eq:trunc-equal}
F_\delta(\sigma)=\frac{\sigma^q}{q}-\frac{\sigma^p}{p},\quad0\le\sigma\le\delta,
\end{equation}
\begin{equation}\label{eq:trunc-convex}
F_\delta(\sigma)\ge0,\quad F_\delta'(\sigma)\ge0,\quad F_\delta''(\sigma)\ge0,
\end{equation}
\begin{equation}\label{eq:trunc-mult}
\sigma F_\delta'(\sigma)-2F_\delta(\sigma)\ge0,
\end{equation}
\begin{equation}\label{eq:trunc-growth}
c_\delta\sigma^q\le F_\delta(\sigma)\le C_\delta\sigma^q,\quad0\le F_\delta'(\sigma)\le C_\delta\sigma^{q-1},\quad
0\le F_\delta''(\sigma)\le C_\delta\sigma^{q-2}.
\end{equation}
\end{lemma}

\begin{proof}
We set
\[
\mathcal M(\sigma)=\sigma F'(\sigma)-2F(\sigma)=\frac{q-2}{q}\sigma^q-\frac{p-2}{p}\sigma^p.
\]
Since $\frac{q-2}{p-2}$ is smaller than $\frac pq$, $1$, $\frac{q-1}{p-1}$, and $\frac{p(q-2)}{q(p-2)}$, we obtain from \eqref{eq:delta-choice} that
\[
F(\sigma)>0,\quad F'(\sigma)>0,\quad F''(\sigma)>0,\quad\mathcal M(\sigma)>0\quad(0<\sigma\le\delta).
\]
Substituting the coefficients in \eqref{eq:trunc-coefficients}, we have
\begin{align*}
A_\delta\delta^q+B_\delta\delta^2+\Gamma_\delta&=\frac{\delta^q}{q}-\frac{\delta^p}{p},\\
qA_\delta\delta^{q-1}+2B_\delta\delta&=\delta^{q-1}-\delta^{p-1},\\
q(q-1)A_\delta\delta^{q-2}+2B_\delta&=(q-1)\delta^{q-2}-(p-1)\delta^{p-2}.
\end{align*}
The one-sided derivatives of $F_\delta$ up to order two
agree at $\sigma=\delta$.
Hence $F_\delta\in C^2([0,\infty))$, while~\eqref{eq:trunc-equal}
follows directly from its definition.

For $\sigma\ge\delta$, we have
\[
F_\delta'(\sigma)=qA_\delta\sigma^{q-1}+2B_\delta\sigma>0,\quad F_\delta''(\sigma)=q(q-1)A_\delta\sigma^{q-2}+2B_\delta>0.
\]
Since $F_\delta(\delta)=F(\delta)>0$ and
$F_\delta'>0$ on $[\delta,\infty)$, we also have
$F_\delta(\sigma)>0$ for $\sigma\ge\delta$.
Combining these inequalities with those on $[0,\delta]$,
we obtain \eqref{eq:trunc-convex}.
Moreover, for $\sigma\ge\delta$,
\[
\sigma F_\delta'(\sigma)-2F_\delta(\sigma)
=(q-2)A_\delta\sigma^q-2\Gamma_\delta>0,
\]
because $A_\delta>0$ and $\Gamma_\delta<0$.
Together with the estimate on $[0,\delta]$, this proves
\eqref{eq:trunc-mult}.

On $[0,\delta]$, by \eqref{eq:trunc-definition}, we deduce that
\[
\left(\frac1q-\frac{\delta^{p-q}}p\right)\sigma^q\le F_\delta(\sigma)\le\frac1q\sigma^q,\quad
0\le F_\delta'(\sigma)\le\sigma^{q-1},\quad0\le F_\delta''(\sigma)\le(q-1)\sigma^{q-2}.
\]
On $[\delta,\infty)$, the upper bounds in
\eqref{eq:trunc-growth} follow directly from
\eqref{eq:trunc-definition}.
For the lower bound, $\frac{F_\delta(\sigma)}{\sigma^q}$ is
positive and continuous on $[\delta,\infty)$ and converges
to $A_\delta>0$ as $\sigma\to\infty$, so its infimum on
this interval is positive. Thus \eqref{eq:trunc-growth} holds for all $\sigma\ge0$.
\end{proof}

By \eqref{eq:trunc-convex}, $F_\delta$ is convex and nondecreasing
on $[0,\infty)$. Hence the function $z\mapsto F_\delta(|z|)$
is convex on $\C^4\simeq\R^8$.

Define
\[
f_\delta(\sigma)=
\begin{cases}
\displaystyle
\frac{F_\delta'(\sigma)}{\sigma},
&\sigma>0,\\
0,
&\sigma=0,
\end{cases}
\]
and set $g_\delta(z)=f_\delta(|z|)z$ for $z\in\C^4$.

\begin{lemma}\label{lem:gdelta-C1}
Regarded as a map from $\R^8$ to $\R^8$, $g_\delta$ is of class $C^1$, $Dg_\delta(0)=0$, and
\begin{equation}\label{eq:gdelta-derivative}
\norm{Dg_\delta(z)}_{\mathrm{op}}\le C_\delta|z|^{q-2},\quad z\in\C^4.
\end{equation}
Consequently,
\begin{equation}\label{eq:gdelta-lipschitz}
\left|g_\delta(z)-g_\delta(y)\right|\le C_\delta\left(|z|^{q-2}+|y|^{q-2}\right)|z-y|,\quad y,z\in\C^4.
\end{equation}
\end{lemma}

\begin{proof}
By \eqref{eq:trunc-growth}, we have $F_\delta'(0)=0$ and
$f_\delta(\sigma)\to0=f_\delta(0)$ as $\sigma\to0^+$.
Differentiating $g_\delta(z)=f_\delta(|z|)z$ at $z\ne0$
in the direction $h\in\C^4\simeq\R^8$, we obtain
\[
Dg_\delta(z)[h]=f_\delta(|z|)h+\frac{f_\delta'(|z|)}{|z|}\Re(z\cdot\overline h)z.
\]
It also follows from \eqref{eq:trunc-growth} that
\[
|f_\delta(\sigma)|+\sigma|f_\delta'(\sigma)|\le C_\delta\sigma^{q-2},\quad\sigma>0.
\]
This proves \eqref{eq:gdelta-derivative} for $z\ne0$. Moreover,
\[
\frac{|g_\delta(z)|}{|z|}=|f_\delta(|z|)|\le C_\delta|z|^{q-2}\rightarrow0\quad\text{as }z\to0.
\]
Thus $Dg_\delta(0)=0$. By \eqref{eq:gdelta-derivative},
the derivative is continuous at zero.
Applying the mean value formula on the segment from $y$ to $z$
and using $|y+t(z-y)|^{q-2}\le C(|y|^{q-2}+|z|^{q-2})$,
we obtain \eqref{eq:gdelta-lipschitz}.
\end{proof}

Define
\[
I_\delta(u)=\frac12\norm{u^+}_E^2-\frac12\norm{u^-}_E^2-\int_{\R^3}|x|^{-b}F_\delta(|u|)\,dx.
\]
By \eqref{eq:trunc-growth}, Lemmas~\ref{lem:weighted} and~\ref{lem:gdelta-C1}, and standard Nemytskii theory, we have $I_\delta\in C^2(E,\R)$, with
\[
I_\delta'(u)[\phi]=\ip{u^+}{\phi^+}_E-\ip{u^-}{\phi^-}_E-\Re\int_{\R^3}|x|^{-b}f_\delta(|u|)u\cdot\overline\phi\,dx.
\]
Writing
\[
\mathcal N_\delta(u)=\int_{\R^3}|x|^{-b}F_\delta(|u|)\,dx,
\]
we have
\begin{equation}\label{eq:N}
\norm{\mathcal N_\delta'(u)}_{E^*}\le C\norm{u}_E^{q-1},\quad\norm{\mathcal N_\delta''(u)}_{\mathscr L(E,E^*)}\le C\norm{u}_E^{q-2}.
\end{equation}
The variational construction below is carried out for
$I_\delta$, whose nonlinear term has subcritical growth.

\section{Constrained spectral reduction}\label{sec:reduction}

\subsection{Fiber maximization} We adapt the mass-preserving spectral reduction of Buffoni and
Jeanjean~\cite[Theorem~2.1]{Buffoni-1993} to the truncated
functional $I_\delta$ with the singular weight $|x|^{-b}$.
For positive spectral components in a fixed cone and with
sufficiently small mass, the weighted embedding and
the derivative bounds for $F_\delta$ yield uniform strict
concavity of the fiber functionals on suitable balls in $E^-$.
We prove that each fiber functional has a unique maximizer
and use these maximizers to define the reduced functional
for the subsequent minimization argument.


Set
\[
\begin{gathered}
S_a=\{u\in E:\norm{u}_2^2=a\},\quad S_a^+=S_a\cap E^+,\\
\mathcal C=\left\{v\in E^+:\norm{v}_E<\sqrt{2(m+1)}\norm{v}_2\right\}.
\end{gathered}
\]
For $v\in\mathcal C\setminus\{0\}$, define
\[
\mathcal B(v)=\left\{w\in E^-:\norm{w}_E\le\frac{\sqrt m}{2}\norm{v}_2\right\},\quad
h(v,w)=\frac{\sqrt{\norm{v}_2^2-\norm{w}_2^2}}{\norm{v}_2}v+w.
\]
By \eqref{eq:normequiv}, every $w\in\mathcal B(v)$ satisfies $\norm{w}_2\le\frac{\norm{v}_2}{2}$, which implies that $h$ is well defined and $\norm{h(v,w)}_2=\norm{v}_2$.

\begin{lemma}\label{lem:fiber}
There exists $a_0>0$ such that, for every $v\in\mathcal C$ with $0<\norm{v}_2^2<a_0$, the fiber functional $\Phi_v(w):=I_\delta(h(v,w))$ on $\mathcal B(v)$ has a unique maximizer $w_\delta(v)$ in the interior of $\mathcal B(v)$. Moreover, the maps $v\mapsto w_\delta(v)$ and
\[
G_\delta(v):=h(v,w_\delta(v))
\]
are of class $C^1$ on $\{v\in\mathcal C:0<\norm{v}_2^2<a_0\}$.
\end{lemma}

\begin{proof}
Fix $v$ as in the statement. Set $\ell=\norm{v}_2$ and $t(w)=\sqrt{1-\frac{\norm{w}_2^2}{\ell^2}}$. For $w\in\mathcal B(v)$,
$\norm{w}_2\le\frac{\ell}{2}$ and hence $t(w)\ge\frac{\sqrt3}{2}$. Since $v\in\mathcal C$ and $\norm{w}_E\le\sqrt m\,\frac{\ell}{2}$, we also have
\[
\norm{h(v,w)}_E\le\norm{v}_E+\norm{w}_E\le C\ell.
\]
For every $z\in E^-$, we deduce that
\begin{align*}
D_wh(v,w)[z]&=z-\frac{\ip{w}{z}_{2}}{\ell\sqrt{\ell^2-\norm{w}_2^2}}v,\\
D_w^2h(v,w)[z,z]&=-\left(\frac{\norm{z}_2^2}{\ell\sqrt{\ell^2-\norm{w}_2^2}}+\frac{\ip{w}{z}_{2}^2}
{\ell(\ell^2-\norm{w}_2^2)^{\frac32}}\right)v.
\end{align*}
Consequently,
\[
\norm{D_wh(v,w)[z]}_E\le C\norm{z}_E,\quad\norm{D_w^2h(v,w)[z,z]}_E\le C\ell^{-1}\norm{z}_E^2.
\]

Define
\[
Q_v(w):=\frac12\left(1-\frac{\norm{w}_2^2}{\ell^2}\right)\norm{v}_E^2-\frac12\norm{w}_E^2.
\]
Then
$\Phi_v(w)=Q_v(w)-\mathcal N_\delta(h(v,w))$, and
\[
D_w^2Q_v(w)[z,z]=-\norm{z}_E^2-\frac{\norm{v}_E^2}{\ell^2}\norm{z}_2^2.
\]
Set $u=h(v,w)$. Applying the chain rule, we obtain
\[
D_w^2(\mathcal N_\delta\circ h)(v,w)[z,z]={}\mathcal N_\delta''(u)[D_wh(v,w)[z],D_wh(v,w)[z]]+\mathcal N_\delta'(u)[D_w^2h(v,w)[z,z]].
\]
By \eqref{eq:N} and the preceding estimates,
\[
\left|\mathcal N_\delta''(u)[D_wh(v,w)[z],D_wh(v,w)[z]]\right|
+\left|\mathcal N_\delta'(u)[D_w^2h(v,w)[z,z]]\right|\le C\ell^{q-2}\norm{z}_E^2.
\]
Therefore,
\[
D_w^2\Phi_v(w)[z,z]\le-\left(1-C\ell^{q-2}\right)\norm{z}_E^2.
\]
The constants above are uniform for $v\in\mathcal C$ and $w\in\mathcal B(v)$. Since $q>2$, we may choose $a_0>0$ sufficiently
small so that
\begin{equation}\label{eq:concave}
D_w^2\Phi_v(w)[z,z]\le-\frac12\norm{z}_E^2,\quad w\in\mathcal B(v),\quad z\in E^-.
\end{equation}

Since $F_\delta\ge0$, we have $\Phi_v(w)\le\frac{\norm{v}_E^2}{2}$ on $\mathcal B(v)$. Thus $M_v:=\sup\limits_{\mathcal B(v)}\Phi_v<\infty$. Let $\{w_n\}\subset\mathcal B(v)$ be a maximizing sequence. Since $\mathcal B(v)$ is convex, $\frac{w_n+w_k}{2}\in\mathcal B(v)$. Applying \eqref{eq:concave} along the segment joining $w_n$ and $w_k$, we obtain
\[
\Phi_v\left(\frac{w_n+w_k}{2}\right)\ge\frac{\Phi_v(w_n)+\Phi_v(w_k)}2+\frac1{16}\norm{w_n-w_k}_E^2.
\]
The left-hand side is at most $M_v$, whereas $\Phi_v(w_n)\to M_v$. Hence $\{w_n\}$ is Cauchy in $E^-$. Since $\mathcal B(v)$ is closed, there exists $\bar w\in\mathcal B(v)$ such that $w_n\to\bar w$ in $E^-$. By continuity, $\Phi_v(\bar w)=M_v$, and the uniqueness of the maximizer follows from \eqref{eq:concave}.

We next show that $\bar w$ lies in the interior of $\mathcal B(v)$. If $w\in\partial\mathcal B(v)$, then $\norm{w}_E^2=\frac{m\ell^2}{4}$ and
\[
\Phi_v(0)-\Phi_v(w)=\frac{\norm{w}_2^2}{2\ell^2}\norm{v}_E^2+\frac12\norm{w}_E^2
+\mathcal N_\delta(h(v,w))-\mathcal N_\delta(v)\ge\frac12\norm{w}_E^2-\mathcal N_\delta(v).
\]
For $v\in\mathcal C$, we have $\mathcal N_\delta(v)\le C\mathcal A_q(v)\le C\ell^q$ by \eqref{eq:trunc-growth} and Lemma~\ref{lem:weighted}. After decreasing $a_0$ if necessary,
\[
\Phi_v(0)-\Phi_v(w)\ge\frac m8\ell^2-C\ell^q\ge\frac m{16}\ell^2>0.
\]
Thus $\Phi_v(w)<\Phi_v(0)$ on $\partial\mathcal B(v)$. It follows that $\bar w$ lies in the interior of $\mathcal B(v)$ and satisfies $D_w\Phi_v(\bar w)=0$. Set $w_\delta(v):=\bar w$ and $G_\delta(v):=h(v,w_\delta(v))$.

We finally prove that the maps $w_\delta$ and $G_\delta$ are of class $C^1$. Consider the open set
\[
\Omega:=\left\{(v,w)\in(E^+\setminus\{0\})\times E^-:\norm{w}_2<\norm{v}_2\right\}.
\]
The map
\[
h(v,w)=\sqrt{1-\frac{\norm{w}_2^2}{\norm{v}_2^2}}\,v+w
\]
is of class $C^2$ on $\Omega$. Since $I_\delta\in C^2(E,\R)$, the map $\Psi:\Omega\to(E^-)^*$ defined by $\Psi(v,w):=D_w\Phi_v(w)$ is of class $C^1$.

At $w=w_\delta(v)$, define
\[
\mathfrak b_v[\xi,z]:=-\left\langle D_w\Psi(v,w)[\xi],z\right\rangle=-D_w^2\Phi_v(w)[\xi,z].
\]
This bilinear form is continuous on $E^-\times E^-$ and, by \eqref{eq:concave}, satisfies $\mathfrak b_v[\xi,\xi]\ge\frac12\norm{\xi}_E^2$.
Applying Lax--Milgram theorem, we conclude that $D_w\Psi(v,w)$ is an isomorphism from $E^-$ onto $(E^-)^*$. From the implicit function theorem, we obtain a neighbourhood $U$ of $v$ in $E^+\setminus\{0\}$ and a $C^1$ map $v'\mapsto\widetilde w(v')$ such that $\Psi(v',\widetilde w(v'))=0$ and $\widetilde w(v)=w_\delta(v)$. 

Since $w_\delta(v)$ lies in the interior of $\mathcal B(v)$,
\[
\norm{w_\delta(v)}_E<\frac{\sqrt m}{2}\norm{v}_2.
\]
After shrinking $U$ if necessary, we have $\norm{\widetilde w(v')}_E<\frac{\sqrt m}{2}\norm{v'}_2$ for all $v'\in U$. Hence, whenever $v'\in U\cap\mathcal C$ and $0<\norm{v'}_2^2<a_0$, $\widetilde w(v')$ lies in the interior of $\mathcal B(v')$. By the strict concavity of $\Phi_{v'}$, we have $\widetilde w(v')=w_\delta(v')$. Thus $w_\delta$ and $G_\delta$ are of class $C^1$ on $\{v\in\mathcal C:0<\norm{v}_2^2<a_0\}$. 
\end{proof}

\begin{lemma}\label{lem:w-estimate}
For every $v\in\mathcal C$ with $\ell=\norm{v}_2$ and $0<\ell^2<a_0$, we have
\begin{equation}\label{eq:w-estimate}
\norm{w_\delta(v)}_E\le C\mathcal A_q(v)^{\frac{q-1}{q}}\le C\ell^{q-1}.
\end{equation}
\end{lemma}

\begin{proof}
Set $w=w_\delta(v)$. By \eqref{eq:concave} and $D_w\Phi_v(w)=0$, we have
\[
0=D_w\Phi_v(w)[w]\le D_w\Phi_v(0)[w]-\frac12\norm{w}_E^2,
\]
which implies
\[
\frac12\norm{w}_E^2\le D_w\Phi_v(0)[w]\le\norm{D_w\Phi_v(0)}_{(E^-)^*}\norm{w}_E.
\]
At $w=0$,
\[
D_w\Phi_v(0)[z]=-\Re\int_{\R^3}|x|^{-b}f_\delta(|v|)v\cdot\overline z\,dx.
\]
Using $|f_\delta(\sigma)\sigma|\le C_\delta\sigma^{q-1}$, H\"{o}lder's inequality with respect to $|x|^{-b}\,dx$, and Lemma~\ref{lem:weighted}, we obtain
\[
\norm{D_w\Phi_v(0)}_{(E^-)^*}\le C\mathcal A_q(v)^{\frac{q-1}{q}}.
\]
Hence
\[
\norm{w}_E\le C\mathcal A_q(v)^{\frac{q-1}{q}}.
\]
Since $v\in\mathcal C$, we have $\norm{v}_E\le C\ell$. It follows from \eqref{eq:weighted} and the above estimates that
\[
\norm{w_\delta(v)}_E\le C\mathcal A_q(v)^{\frac{q-1}{q}}\le C\norm{v}_E^{q-1}\le C\ell^{q-1}.
\]
\end{proof}

\subsection{The reduced functional and the lifting estimate}
For $v\in\mathcal C$ with $0<\norm{v}_2^2<a_0$, we define
the reduced functional by
\[
J_\delta(v)=I_\delta(G_\delta(v)).
\]
Fix $a\in(0,a_0)$ for the remainder of this section.
For $u\in S_a$ with $u^+\ne0$, set
\[
\kappa_\delta(u)
=
\frac{
\norm{u^+}_E^2
-\Re\displaystyle\int_{\R^3}
|x|^{-b}f_\delta(|u|)u\cdot\overline{u^+}\,dx
}{
\norm{u^+}_2^2
},
\]
and
\[
\mathcal H_\delta(u)
=
H_0u-|x|^{-b}f_\delta(|u|)u-\kappa_\delta(u)u
\in E^*.
\]
By the definition of $\kappa_\delta(u)$, we have
\[
\langle\mathcal H_\delta(u),u^+\rangle=0.
\]

For $v\in S_a^+\cap\mathcal C$, write
\[
w=w_\delta(v),\quad
t=\sqrt{1-\frac{\norm{w}_2^2}{a}},\quad
u=G_\delta(v)=tv+w.
\]
Then $t\ge\frac{\sqrt3}{2}$. From $D_w\Phi_v(w)=0$
and $\langle\mathcal H_\delta(u),u^+\rangle=0$, we obtain
\[
\langle\mathcal H_\delta(u),z\rangle=0
\quad\text{for every }z\in E^-.
\]

For $v\in S_a^+$, let
\[
T_vS_a^+
=
\left\{\xi\in E^+:\ip{v}{\xi}_{2}=0\right\}
\]
denote the tangent space of $S_a^+$ at $v$.
We next estimate $\mathcal H_\delta(G_\delta(v))$ in terms of the derivative of $J_\delta$ on $S_a^+$.

\begin{lemma}\label{lem:lifting}
There exists a constant $C>0$, independent of $a\in(0,a_0)$ and $v\in S_a^+\cap\mathcal C$, such that
\begin{equation}\label{eq:lifting}
\norm{\mathcal H_\delta(G_\delta(v))}_{E^*}\le C\left\|d\bigl(J_\delta|_{S_a^+}\bigr)(v)\right\|_{(T_vS_a^+)^*}.
\end{equation}
\end{lemma}

\begin{proof}
Set $u=G_\delta(v)=tv+w$, where $w=w_\delta(v)$ and $t=\sqrt{1-\frac{\norm{w}_2^2}{a}}$. Let $\kappa=\kappa_\delta(u)$ and define $\mathscr L_u\in E^*$ by
\[
\langle\mathscr L_u,\phi\rangle=\langle I_\delta'(u),\phi\rangle-\kappa\ip{u}{\phi}_{2},\quad \phi\in E.
\]
Then $\mathscr L_u=\mathcal H_\delta(u)$. By the definition of $\kappa_\delta(u)$, we deduce that
$\langle\mathscr L_u,u^+\rangle=0$. Since $u^+=tv$ and $t>0$, it follows that $\langle\mathscr L_u,v\rangle=0$.

Since $w$ is an interior maximizer of the fiber functional $\Phi_v$, we have $D_w\Phi_v(w)=0$, and hence
\[
\left\langle I_\delta'(u),D_wh(v,w)[z]\right\rangle=0\quad\text{for every }z\in E^-,
\]
where
\[
D_wh(v,w)[z]=z-\frac{\ip{w}{z}_2}{at}v.
\]
By the $L^2$-orthogonality of $E^+$ and $E^-$,
\[
\ip{u}{D_wh(v,w)[z]}_2=\ip{w}{z}_2-\frac{\ip{w}{z}_2}{at}\ip{tv}{v}_2=0.
\]
Thus $\langle\mathscr L_u,D_wh(v,w)[z]\rangle=0$. Since $\langle\mathscr L_u,v\rangle=0$, we obtain
$\langle\mathscr L_u,z\rangle=0$ for every $z\in E^-$.

Let $\xi\in T_vS_a^+$. Differentiating
$\norm{G_\delta(v)}_2^2=a$ and $(G_\delta(v))^+=tv$,
we deduce that
\[
\ip{u}{G_\delta'(v)[\xi]}_2=0,
\quad
\bigl(G_\delta'(v)[\xi]\bigr)^+
=t\xi+c_\xi v,
\quad c_\xi\in\R.
\] 
Since $\langle\mathscr L_u,v\rangle=0$ and $\langle\mathscr L_u,z\rangle=0$ for every $z\in E^-$, we have
$\langle\mathscr L_u,G_\delta'(v)[\xi]\rangle=t\langle\mathscr L_u,\xi\rangle$.
By the chain rule and $\ip{u}{G_\delta'(v)[\xi]}_2=0$, we obtain
\[
d\bigl(J_\delta|_{S_a^+}\bigr)(v)[\xi]=\langle I_\delta'(u),G_\delta'(v)[\xi]\rangle=t\langle\mathscr L_u,\xi\rangle.
\]

For arbitrary $\phi\in E$, write
\[
\phi^+=\xi+cv,\quad c=\frac{\ip{v}{\phi^+}_2}{a},\quad\xi=\phi^+-cv\in T_vS_a^+.
\]
Thus $\langle\mathscr L_u,\phi\rangle=\langle\mathscr L_u,\xi\rangle$. Moreover, $v\in S_a^+\cap\mathcal C$ satisfies
$\norm{v}_2=\sqrt a$ and
$\norm{v}_E\le\sqrt{2(m+1)a}$. By \eqref{eq:normequiv},
$\norm{\phi^+}_2\le m^{-\frac12}\norm{\phi}_E$. Consequently,
\[
|c|\norm{v}_E\le\frac{\norm{v}_2\norm{\phi^+}_2}{a}\norm{v}_E\le C\norm{\phi}_E,
\]
and hence
\[
\norm{\xi}_E\le\norm{\phi^+}_E+|c|\norm{v}_E\le C\norm{\phi}_E.
\]
Since $t\ge\frac{\sqrt3}{2}$, we have
\begin{align*}
\left|\langle\mathscr L_u,\phi\rangle\right|&=\frac1t\left|d\bigl(J_\delta|_{S_a^+}\bigr)(v)[\xi]\right|\\
&\le\frac2{\sqrt3}\left\|d\bigl(J_\delta|_{S_a^+}\bigr)(v)\right\|_{(T_vS_a^+)^*}\norm{\xi}_E\\
&\le C\left\|d\bigl(J_\delta|_{S_a^+}\bigr)(v)\right\|_{(T_vS_a^+)^*}\norm{\phi}_E.
\end{align*}
Taking the supremum over all $\phi\in E$ with $\norm{\phi}_E\le1$, we obtain \eqref{eq:lifting}.
\end{proof}

\section{Existence for the truncated problem}\label{sec:truncated}
We now seek a minimizer of $J_\delta$ on a closed subset
of $S_a^+\cap\mathcal C$. We first show that its infimum
is strictly below the values on the boundary of this set,
so that minimizing sequences eventually lie in its interior.

For $\chi>0$ set
\[
\mathcal V_\chi^+=\{v\in E^+:\norm{v}_E^2<\chi\},\quad\overline{\mathcal V}_\chi^+=\{v\in E^+:\norm{v}_E^2\le\chi\},
\quad\Theta_a=2(m+1)a,
\]
so that
\[
S_a^+\cap\mathcal C=S_a^+\cap\mathcal V_{\Theta_a}^+.
\]

\subsection{Energy bounds and the minimization set} 
The following lemma estimates $\mathcal A_r(v)$ in terms of the two-component function $\omega$ in the Foldy--Wouthuysen representation of $v$.


\begin{lemma}\label{lem:FW-low}
Let $\omega\in H^1(\R^3;\C^2)$ and $v= U_{\rm FW}^{-1}\begin{pmatrix}\omega \\0\end{pmatrix}\in E^+$. Then, for every $r\in(2,3-b]$,
\begin{equation}\label{eq:FW-weighted-upper}
\mathcal A_r(v) \le C\bigl(\mathcal A_r(\omega)+\norm{\nabla\omega}_2^r\bigr).
\end{equation}
\end{lemma}

\begin{proof}
Using $U_{\rm FW}=\mathcal F^{-1}U\mathcal F$
and the explicit formula for $U(\xi)$, we obtain
\[
\widehat v(\xi)
=U(\xi)^{-1}\binom{\widehat\omega(\xi)}{0}
=\begin{pmatrix}
\rho_+(\xi)\widehat\omega(\xi)\\
\rho_-(\xi)S(\xi)\widehat\omega(\xi)
\end{pmatrix},
\]
where
\[
S(\xi)=\frac{\boldsymbol{\sigma}\cdot\xi}{|\xi|}\quad\text{for }\xi\ne0,\quad S(0)=0.
\]
Since
$(\boldsymbol{\sigma}\cdot\xi)^2=|\xi|^2I_2$, the matrix $S(\xi)$ is unitary for $\xi\ne0$.
Let $\mathscr R\omega=\mathcal F^{-1}[(1-\rho_+)\widehat\omega]$. Then the upper two-component spinor of $v$ is $\omega-\mathscr R\omega$. Moreover,
\[
1-\rho_+(\xi)=\frac{\rho_-^2(\xi)}{1+\rho_+(\xi)}\le\rho_-^2(\xi),
\quad\lambda(\xi)\rho_-^2(\xi)=\frac{\lambda(\xi)-m}{2}\le\frac{|\xi|^2}{4m}.
\]
Since $0\le\rho_-^2(\xi)\le1$, we also have $(1-\rho_+(\xi))^2\le\rho_-^2(\xi)$. Hence, by Plancherel's theorem and $\lambda(\xi)\asymp(1+|\xi|^2)^{\frac12}$,
\[
\norm{\mathscr R\omega}_{H^{\frac12}}^2+\left\|\mathcal F^{-1}[\rho_-S\widehat\omega]\right\|_{H^{\frac12}}^2
\le C\norm{\nabla\omega}_2^2.
\]
By Lemma~\ref{lem:weighted}, we obtain
\[
\mathcal A_r(\mathscr R\omega)+\mathcal A_r\left(\mathcal F^{-1}[\rho_-S\widehat\omega]\right)\le C\norm{\nabla\omega}_2^r.
\]
Using $(A^2+B^2)^{\frac r2}\le C_r(A^r+B^r)$ for $A,B\ge0$ and
$|\omega-\mathscr R\omega|^r\le C_r(|\omega|^r+|\mathscr R\omega|^r)$, we conclude that
\[
\mathcal A_r(v)\le C\left[\mathcal A_r(\omega)+\mathcal A_r(\mathscr R\omega)+\mathcal A_r
\left(\mathcal F^{-1}[\rho_-S\widehat\omega]\right)\right]\le C\bigl(\mathcal A_r(\omega)+\norm{\nabla\omega}_2^r\bigr),
\]
which proves \eqref{eq:FW-weighted-upper}.
\end{proof}

The next estimate extends the argument of
\cite[Lemma~3.5]{WangChang2026} to the range
$2<r\leq 3-b$.
\begin{lemma}\label{lem:comparison}
Let $r\in(2,3-b]$, $v\in S_a^+$, and $\zeta\in E^-$ satisfy $\norm{\zeta}_2\le\frac{\sqrt a}{2}$. Set $t=\sqrt{1-\frac{\norm{\zeta}_2^2}{a}}$ and $u=tv+\zeta$. Then there exists a constant $C>0$, independent of $a$, $v$, and $\zeta$, such that
\begin{equation}\label{eq:comparison-v}
\mathcal A_r(u)\ge 2^{1-r}\mathcal A_r(v)-\frac{C}{a}\norm{\zeta}_2^2\norm{v}_E^r-C\norm{\zeta}_E^r.
\end{equation}
Moreover, if
$v=U_{\rm FW}^{-1}\binom{\omega}{0}$ for some $\omega\in H^1(\R^3;\C^2)$, then
\begin{equation}\label{eq:comparison-w}
\mathcal A_r(u)\ge 4^{1-r}\mathcal A_r(\omega)-C\norm{\nabla\omega}_2^r-\frac{C}{a}\norm{\zeta}_2^2\norm{v}_E^r
-C\norm{\zeta}_E^r.
\end{equation}
\end{lemma}

\begin{proof}
Since $\frac{\norm{\zeta}_2^2}{a}\le\frac14$, we have $t^r\ge1-C\frac{\norm{\zeta}_2^2}{a}$. By the pointwise inequality $|A+B|^r\ge2^{1-r}|A|^r-|B|^r$
with $A=tv$ and $B=\zeta$, we have
\[
\mathcal A_r(u)\ge2^{1-r}\mathcal A_r(v)-C\frac{\norm{\zeta}_2^2}{a}\mathcal A_r(v)-\mathcal A_r(\zeta).
\]
Applying Lemma~\ref{lem:weighted} to $\mathcal A_r(v)$ and $\mathcal A_r(\zeta)$ in the two negative terms, we obtain
\eqref{eq:comparison-v}.

As in the proof of Lemma~\ref{lem:FW-low}, the upper two-component spinor of $v$ is $\omega-\mathscr R\omega$, and
$\mathcal A_r(\mathscr R\omega)\le C\norm{\nabla\omega}_2^r$. Hence
\[
\mathcal A_r(v)\ge\mathcal A_r(\omega-\mathscr R\omega)\ge2^{1-r}\mathcal A_r(\omega)-C\norm{\nabla\omega}_2^r.
\]
Substituting this estimate into \eqref{eq:comparison-v}, we deduce \eqref{eq:comparison-w}.
\end{proof}

Recall that $\gamma_q=\frac{3(q-2)}{2}+b$. Since $2<q<3-b$, we have
\begin{equation}\label{eq:gamma-range}
0<\gamma_q<2,\quad q>\gamma_q.
\end{equation}

\begin{lemma}\label{lem:annulus}
There exists $a_1\in(0,a_0)$ such that, for every $a\in(0,a_1)$,
\begin{equation}\label{eq:threshold-chain}
\inf_{v\in S_a^+\cap\mathcal V_{\frac{\Theta_a}{2}}^+}J_\delta(v)<\frac{ma}{2}<
\inf_{v\in S_a^+\cap(\mathcal V_{\Theta_a}^+\setminus\mathcal V_{\frac{3\Theta_a}{4}}^+)}J_\delta(v).
\end{equation}
\end{lemma}

\begin{proof} We prove the first inequality in~(\ref{eq:threshold-chain}).
Fix $\varphi_0\in C_c^\infty(\R^3;\C^2)\setminus\{0\}$ with
$\norm{\varphi_0}_2=1$. For $\varepsilon\in(0,1)$, define
\[
\varphi_\varepsilon(x)=\sqrt a\,\varepsilon^{\frac32}\varphi_0(\varepsilon x),\quad
v_\varepsilon=U_{\rm FW}^{-1}\binom{\varphi_\varepsilon}{0}.
\]
Since $U_{\rm FW}$ is unitary, we have
$\norm{v_\varepsilon}_2^2=\norm{\varphi_\varepsilon}_2^2=a$, and hence $v_\varepsilon\in S_a^+$.
Moreover,
\begin{equation}\label{eq:test-kinetic}
\norm{v_\varepsilon}_E^2=\int_{\R^3}\lambda(\xi)|\widehat\varphi_\varepsilon(\xi)|^2\,d\xi\le ma+C_{\varphi_0}a\varepsilon^2,
\end{equation}
where $C_{\varphi_0}=\frac{\norm{\nabla\varphi_0}_2^2}{2m}$. The scaling identities are
\begin{equation}\label{eq:test-scaling}
\mathcal A_q(\varphi_\varepsilon)=a^{\frac q2}\varepsilon^{\gamma_q}\mathcal A_q(\varphi_0),\quad
\norm{\nabla\varphi_\varepsilon}_2^q=a^{\frac q2}\varepsilon^q\norm{\nabla\varphi_0}_2^q.
\end{equation}
 It follows from Lemma~\ref{lem:FW-low}, \eqref{eq:gamma-range}, and \eqref{eq:test-scaling} that
\begin{equation}\label{eq:test-weighted-upper}
\mathcal A_q(v_\varepsilon)\le Ca^{\frac q2}
\left(\varepsilon^{\gamma_q}+\varepsilon^q\right)\le Ca^{\frac q2}\varepsilon^{\gamma_q}.
\end{equation}

Choose $\varepsilon_0\in(0,1)$ such that $C_{\varphi_0}\varepsilon_0^2<1$. For $0<\varepsilon<\varepsilon_0$, we obtain from \eqref{eq:test-kinetic}
\[
\norm{v_\varepsilon}_E^2<(m+1)a=\frac{\Theta_a}{2}.
\]
Hence $v_\varepsilon\in S_a^+\cap\mathcal V_{\frac{\Theta_a}{2}}^+\subset S_a^+\cap\mathcal C$. Set $\zeta_\varepsilon=w_\delta(v_\varepsilon)$, $t_\varepsilon=\sqrt{1-\frac{\norm{\zeta_\varepsilon}_2^2}{a}}$, and $u_\varepsilon=G_\delta(v_\varepsilon)=t_\varepsilon v_\varepsilon+\zeta_\varepsilon$.

Let $\vartheta_q=\frac{(q-1)\gamma_q}{q}$. By \eqref{eq:trunc-growth}, Lemma~\ref{lem:w-estimate}, and \eqref{eq:test-weighted-upper}, we have
\[
\norm{\zeta_\varepsilon}_E\le C\mathcal A_q(v_\varepsilon)^{\frac{q-1}{q}}\le Ca^{\frac{q-1}{2}}\varepsilon^{\vartheta_q}.
\]
Since $q>2$,
\begin{equation}\label{eq:vartheta-order}
2\vartheta_q>\gamma_q,\quad q\vartheta_q>\gamma_q.
\end{equation}
Applying \eqref{eq:comparison-w} with $r=q$ and using \eqref{eq:test-scaling}--\eqref{eq:vartheta-order}, we deduce that
\[
\mathcal A_q(u_\varepsilon)\ge{}c_{\varphi_0}a^{\frac q2}\varepsilon^{\gamma_q}-Ca^{\frac q2}\varepsilon^q
-Ca^{\frac{3q}{2}-2}\varepsilon^{2\vartheta_q}-Ca^{\frac{q(q-1)}2}\varepsilon^{q\vartheta_q},
\]
where
$c_{\varphi_0}=4^{1-q}\mathcal A_q(\varphi_0)>0$. By \eqref{eq:gamma-range} and \eqref{eq:vartheta-order}, the exponents of $\varepsilon$ in the three error terms are all greater than $\gamma_q$. Hence, for every fixed $a\in(0,a_0)$, there exists $\varepsilon_1(a)\in(0,\varepsilon_0)$ such that
\[
\mathcal A_q(u_\varepsilon)\ge\frac{c_{\varphi_0}}2 a^{\frac q2}\varepsilon^{\gamma_q}
\]
whenever $0<\varepsilon<\varepsilon_1(a)$.

Since $F_\delta(\sigma)\ge c_\delta\sigma^q$ and
\[
\frac12\left(t_\varepsilon^2\norm{v_\varepsilon}_E^2-\norm{\zeta_\varepsilon}_E^2\right)\le\frac12\norm{v_\varepsilon}_E^2,
\]
we obtain
\[
J_\delta(v_\varepsilon)=I_\delta(u_\varepsilon)\le\frac{ma}{2}+C_{\varphi_0}a\varepsilon^2
-c_{\mathrm{test}}a^{\frac q2}\varepsilon^{\gamma_q}
\]
for some $c_{\mathrm{test}}>0$ independent of $a$ and $\varepsilon$. Since $\gamma_q<2$, for each fixed $a\in(0,a_0)$ we may choose $\varepsilon_2(a)\in(0,\varepsilon_1(a))$ such that
\[
C_{\varphi_0}a\varepsilon_2(a)^2<c_{\mathrm{test}}a^{\frac q2}\varepsilon_2(a)^{\gamma_q}.
\]
With this choice of $\varepsilon_2(a)$, we have
\begin{equation}\label{eq:inner-strict-upper}
\inf_{v\in S_a^+\cap\mathcal V_{\frac{\Theta_a}{2}}^+}J_\delta(v)\le J_\delta(v_{\varepsilon_2(a)})<\frac{ma}{2}.
\end{equation}

We now prove the second inequality in (\ref{eq:threshold-chain}).
Let
\[
v\in S_a^+\cap\left(\mathcal V_{\Theta_a}^+\setminus\mathcal V_{\frac{3\Theta_a}{4}}^+\right).
\]
Then
\begin{equation}\label{eq:belt-norm-range}
\frac{3\Theta_a}{4}\le\norm{v}_E^2<\Theta_a.
\end{equation}
Taking $w=0$ in the fiber maximization, we obtain
\[
J_\delta(v)\ge I_\delta(v)=\frac12\norm{v}_E^2-\mathcal N_\delta(v).
\]
From \eqref{eq:trunc-growth} and Lemma~\ref{lem:weighted}, we have $\mathcal N_\delta(v)\le C\norm{v}_E^q$.
Using \eqref{eq:belt-norm-range} and $\Theta_a=2(m+1)a$, we deduce
\begin{equation}\label{eq:belt-strict-lower}
J_\delta(v)\ge\frac12\norm{v}_E^2-C\norm{v}_E^q\ge\frac{3\Theta_a}{8}-Ca^{\frac q2}
=\frac{3(m+1)}4a-Ca^{\frac q2}>\frac{ma}{2}
\end{equation}
for all sufficiently small $a>0$. Here we used $q>2$ and $\frac{3(m+1)}{4}-\frac{m}{2}=\frac{m+3}{4}>0$. Taking the infimum over the outer region, we conclude that
\begin{equation}\label{eq:belt-inf-lower}
\inf_{v\in S_a^+\cap(\mathcal V_{\Theta_a}^+\setminus\mathcal V_{\frac{3\Theta_a}{4}}^+)}J_\delta(v)>\frac{ma}{2}.
\end{equation}
Choosing $a_1\in(0,a_0)$ sufficiently small and combining \eqref{eq:inner-strict-upper} with
\eqref{eq:belt-inf-lower}, we obtain \eqref{eq:threshold-chain}.
\end{proof}

Define
\[
e_a^\delta:=\inf_{v\in S_a^+\cap\mathcal V_{\Theta_a}^+}J_\delta(v),\quad
X_a:=S_a^+\cap\overline{\mathcal V}_{\frac{3\Theta_a}{4}}^+.
\]
Since $\frac{3\Theta_a}{4}<\Theta_a$, the closed set $X_a$ is contained in $S_a^+\cap\mathcal C$, the domain of $J_\delta$ on the mass sphere.

\begin{lemma}\label{lem:energy-level}
For every $a\in(0,a_1)$,
\begin{equation}\label{eq:energy-level-chain}
0<e_a^\delta=\inf_{v\in X_a}J_\delta(v)<\frac{ma}{2}<\inf_{v\in S_a^+\cap\partial\mathcal V_{\frac{3\Theta_a}{4}}^+}J_\delta(v).
\end{equation}
In particular, every minimizing sequence $\{v_n\}$ for $J_\delta$ on $X_a$ satisfies
\[
\norm{v_n}_E^2<\frac{3\Theta_a}{4}\quad\text{for all sufficiently large }n.
\]
\end{lemma}

\begin{proof}
Since $S_a^+\cap\mathcal V_{\frac{\Theta_a}{2}}^+\subset X_a$, we obtain from Lemma~\ref{lem:annulus}
\[
\inf_{v\in X_a}J_\delta(v)\le\inf_{v\in S_a^+\cap\mathcal V_{\frac{\Theta_a}{2}}^+}J_\delta(v)<\frac{ma}{2}
<\inf_{v\in S_a^+\cap(\mathcal V_{\Theta_a}^+\setminus\mathcal V_{\frac{3\Theta_a}{4}}^+)}J_\delta(v).
\]
It follows from the definition of $e_a^\delta$ and the above inequalities that
\[
e_a^\delta=\inf_{v\in X_a}J_\delta(v)<\frac{ma}{2}.
\]
Moreover, for every $v\in S_a^+\cap\mathcal V_{\Theta_a}^+$, we take $w=0$ in the fiber maximization to obtain
\[
J_\delta(v)\ge I_\delta(v)=\frac12\norm{v}_E^2-\mathcal N_\delta(v).
\]
Since $v\in E^+$, $\norm{v}_2^2=a$, and $\norm{v}_E^2<\Theta_a$, we have $\norm{v}_E^2\ge ma$ and
$\mathcal N_\delta(v)\le C\norm{v}_E^q\le C a^{\frac q2}$. Therefore,
\[
J_\delta(v)\ge\frac{ma}{2}-C a^{\frac q2}>0
\]
after decreasing $a_1$ if necessary, which implies $e_a^\delta>0$. We also have
\[
S_a^+\cap\partial\mathcal V_{\frac{3\Theta_a}{4}}^+\subset S_a^+\cap
(\mathcal V_{\Theta_a}^+\setminus\mathcal V_{\frac{3\Theta_a}{4}}^+),
\]
so, by \eqref{eq:belt-strict-lower}, we deduce that
\[
\inf_{v\in S_a^+\cap\partial\mathcal V_{\frac{3\Theta_a}{4}}^+}J_\delta(v)>\frac{ma}{2}>e_a^\delta,
\]
which proves \eqref{eq:energy-level-chain}. Consequently, if $\{v_n\}\subset X_a$ is a minimizing sequence for $J_\delta$, then $\norm{v_n}_E^2<\frac{3\Theta_a}{4}$ for all sufficiently large $n$.
\end{proof}

\subsection{A minimizer of the reduced functional}
By Lemma~\ref{lem:energy-level} and Ekeland's variational
principle, we obtain a Palais--Smale sequence for $J_\delta$
on $S_a^+$ at the level $e_a^\delta$. 

\begin{lemma}\label{lem:ekeland}
For every fixed $a\in(0,a_1)$, there exists a sequence $\{v_n\}\subset X_a$ such that
\begin{equation}\label{eq:reduced-PS}
J_\delta(v_n)\rightarrow e_a^\delta,\quad\norm{d\bigl(J_\delta|_{S_a^+}\bigr)(v_n)}_{(T_{v_n}S_a^+)^*}\rightarrow0,
\end{equation}
and $\norm{v_n}_E^2<\frac{3\Theta_a}{4}$ for all sufficiently large $n$.
\end{lemma}

\begin{proof}
Since $X_a$ is closed in $E^+$, it is complete with respect to the norm $\norm{\cdot}_E$. For each $n$, choose
$\widetilde v_n\in X_a$ such that
\[
J_\delta(\widetilde v_n)\le e_a^\delta+\frac1{n^2}.
\]
By Ekeland's variational principle, there exists $v_n\in X_a$ such that $J_\delta(v_n)\le e_a^\delta+\frac1{n^2}$ and
\begin{equation}\label{eq:ekeland-ineq}
J_\delta(z)\ge J_\delta(v_n)-\frac1n\norm{z-v_n}_E,\quad z\in X_a.
\end{equation}
Thus $J_\delta(v_n)\to e_a^\delta$. By Lemma~\ref{lem:energy-level}, we also have $\norm{v_n}_E^2<\frac{3\Theta_a}{4}$ for all sufficiently large $n$.

Fix such an $n$ and let $\xi\in T_{v_n}S_a^+$. Consider the curve
\[
\mathfrak c_{n,\xi}(\varsigma)=\sqrt a\,\frac{v_n+\varsigma\xi}{\norm{v_n+\varsigma\xi}_2}.
\]
This is a $C^1$ curve in $S_a^+$ with $\mathfrak c_{n,\xi}(0)=v_n$ and $\mathfrak c_{n,\xi}'(0)=\xi$. Since $\norm{v_n}_E^2<\frac{3\Theta_a}{4}$, we have
$\mathfrak c_{n,\xi}(\varsigma)\in X_a$ for all sufficiently small $|\varsigma|$. Applying \eqref{eq:ekeland-ineq} along the curve $\mathfrak c_{n,\xi}$
and using $\mathfrak c_{n,\xi}'(0)=\xi$, we deduce that
\[
\left|d\bigl(J_\delta|_{S_a^+}\bigr)(v_n)[\xi]\right|\le\frac1n\norm{\xi}_E.
\]
Taking the supremum over all
$\xi\in T_{v_n}S_a^+$ with $\norm{\xi}_E\le1$, we obtain
\eqref{eq:reduced-PS}.
\end{proof}

Set
\[
u_n=G_\delta(v_n),\quad \mu_n=\kappa_\delta(u_n).
\]
By Lemma~\ref{lem:lifting}, we have
\begin{equation}\label{eq:PS-full}
\mathcal H_\delta(u_n)\rightarrow0\quad\text{in }E^*,
\end{equation}
and, since $v_n\in X_a$, we obtain from Lemma~\ref{lem:w-estimate}
\begin{equation}\label{eq:PS-bound}
\norm{u_n}_E\le C\sqrt a,\quad \norm{u_n}_2^2=a.
\end{equation}

\begin{lemma}\label{lem:multiplier}
There exists $a_2\in(0,a_1)$ such that, for every fixed $a\in(0,a_2)$,
\begin{equation}\label{eq:multiplier-gap}
0<m-C a^{\frac{q-2}{2}}\le\liminf_{n\to\infty}\mu_n\le\limsup_{n\to\infty}\mu_n\le\frac{2e_a^\delta}{a}<m.
\end{equation}
\end{lemma}

\begin{proof}
Write
\[
u_n^+=t_n v_n,\quad
t_n=\frac{\sqrt{a-\norm{w_\delta(v_n)}_2^2}}{\sqrt a}
\ge\frac{\sqrt3}{2}.
\]
Therefore,
\[
\norm{u_n^+}_2^2=t_n^2a\ge\frac{3a}{4}.
\]
Since $\mu_n=\kappa_\delta(u_n)$,
\[
\mu_n\norm{u_n^+}_2^2=\norm{u_n^+}_E^2-\Re\int_{\R^3}|x|^{-b}f_\delta(|u_n|)u_n\cdot\overline{u_n^+}\,dx.
\]
By Lemma~\ref{lem:weighted}, \eqref{eq:trunc-growth}, and \eqref{eq:PS-bound}, the nonlinear integral satisfies
\[
\left|\int_{\R^3}|x|^{-b}f_\delta(|u_n|)u_n\cdot\overline{u_n^+}\,dx\right|\le C\norm{u_n}_{L_b^q}^{q-1}\norm{u_n^+}_{L_b^q}\le C a^{\frac q2}.
\]
Since $\norm{u_n^+}_E^2\ge m\norm{u_n^+}_2^2$, it follows that
\[
\mu_n\ge m-Ca^{\frac{q-2}{2}}.
\]
Choose $a_2\in(0,a_1)$ sufficiently small so that $m-Ca^{\frac{q-2}{2}}>0$ for every $a\in(0,a_2)$. This proves the lower bound in \eqref{eq:multiplier-gap}.

Testing \eqref{eq:PS-full} against $u_n$ and using \eqref{eq:PS-bound}, we obtain
\[
\mu_n a=2I_\delta(u_n)-\int_{\R^3}|x|^{-b}\left(|u_n|F_\delta'(|u_n|)-2F_\delta(|u_n|)\right)\,dx+o_n(1).
\]
By \eqref{eq:trunc-mult}, it follows that
\[
\mu_n a\le2I_\delta(u_n)+o_n(1)=2J_\delta(v_n)+o_n(1).
\]
Passing to the upper limit and using~\eqref{eq:energy-level-chain}, we conclude that
\[
\limsup_{n\to\infty}\mu_n a\le2e_a^\delta<ma.
\]
Together with the lower bound, this proves \eqref{eq:multiplier-gap}.
\end{proof}

\begin{proposition}\label{prop:PS}
For every $a\in(0,a_2)$, let $\{v_n\}$ be the sequence obtained in Lemma~\ref{lem:ekeland}. Then there exist
\[
v_a\in X_a,\quad \mu_a^\delta\in(0,m)
\]
such that, up to a subsequence,
\begin{equation*}\label{eq:PS-convergence}
v_n\rightarrow v_a\quad\text{in }E^+,\quad u_n\rightarrow u_a^\delta:=G_\delta(v_a)\quad\text{in }E,\quad\mu_n\rightarrow\mu_a^\delta\quad\text{in }\R.
\end{equation*}
Moreover,
\begin{equation}\label{eq:truncated-solution}
H_0u_a^\delta=\mu_a^\delta u_a^\delta+|x|^{-b}f_\delta(|u_a^\delta|)u_a^\delta\quad\text{in }E^*,\quad\norm{u_a^\delta}_2^2=a,
\end{equation}
and
\begin{equation}\label{eq:truncated-properties}
0<J_\delta(v_a)=I_\delta(u_a^\delta)=e_a^\delta<\frac{ma}{2},\quad\norm{u_a^\delta}_E\le C_0\sqrt a,
\end{equation}
where $C_0>0$ is independent of $a$.
\end{proposition}

\begin{proof}
By \eqref{eq:PS-bound} and Lemma~\ref{lem:multiplier}, we may pass to a subsequence such that $u_n\rightharpoonup u_a^\delta$ in $E$ and $\mu_n\to\mu_a^\delta$ for some $\mu_a^\delta\in(0,m)$.

Since $q<3-b$, we obtain from  Lemma~\ref{lem:weighted} that
\[
u_n\rightarrow u_a^\delta\quad\text{in }L^q(\R^3,|x|^{-b}\,dx;\C^4).
\]
By \eqref{eq:gdelta-lipschitz} and H\"{o}lder's inequality, for every $\phi\in E$,
\[
\begin{aligned}
&\left|\int_{\R^3}|x|^{-b}\left(f_\delta(|u_n|)u_n-f_\delta(|u_a^\delta|)u_a^\delta\right)\cdot\overline\phi\,dx\right|\\
&\quad\le C\left(\norm{u_n}_{L_b^q}^{q-2}+\norm{u_a^\delta}_{L_b^q}^{q-2}\right)\norm{u_n-u_a^\delta}_{L_b^q}\norm{\phi}_{L_b^q}.
\end{aligned}
\]
Consequently,
\begin{equation}\label{eq:nonlinear-Edual}
|x|^{-b}f_\delta(|u_n|)u_n\rightarrow|x|^{-b}f_\delta(|u_a^\delta|)u_a^\delta\quad\text{in }E^*.
\end{equation}
Passing to the limit in \eqref{eq:PS-full}, we conclude that
\[
H_0u_a^\delta=\mu_a^\delta u_a^\delta+|x|^{-b}f_\delta(|u_a^\delta|)u_a^\delta\quad\text{in }E^*.
\]

Set $z_n=u_n-u_a^\delta$. Since $\{u_n\}$ is bounded in $E$ and $\mu_n\to\mu_a^\delta$, we have $(\mu_n-\mu_a^\delta)u_n\to0$ in $E^*$. It follows from \eqref{eq:PS-full}, \eqref{eq:nonlinear-Edual}, and the equation for $u_a^\delta$ that
\begin{equation}\label{eq:gap-difference}
(H_0-\mu_a^\delta)z_n\rightarrow0\quad\text{in }E^*.
\end{equation}

Testing \eqref{eq:gap-difference} against $z_n^+$ and $-z_n^-$, respectively, we obtain
\[
\norm{z_n^+}_E^2-\mu_a^\delta\norm{z_n^+}_2^2=o_n(1)\norm{z_n^+}_E
\]
and
\[
\norm{z_n^-}_E^2+\mu_a^\delta\norm{z_n^-}_2^2=o_n(1)\norm{z_n^-}_E.
\]
Using $\norm{z_n^+}_E^2\ge m\norm{z_n^+}_2^2$ and $0<\mu_a^\delta<m$, we deduce
\[
\left(1-\frac{\mu_a^\delta}{m}\right)\norm{z_n^+}_E^2\le o_n(1)\norm{z_n^+}_E,\quad\norm{z_n^-}_E^2\le o_n(1)\norm{z_n^-}_E.
\]
Hence $u_n\to u_a^\delta$ strongly in $E$.

Since $I_\delta(u_n)=J_\delta(v_n)\to e_a^\delta$ and $u_n\to u_a^\delta$ strongly in $E$, we have $\norm{u_a^\delta}_2^2=a$ and $I_\delta(u_a^\delta)=e_a^\delta$. Hence, by Lemma~\ref{lem:energy-level},
\[
0<I_\delta(u_a^\delta)=e_a^\delta<\frac{ma}{2}.
\]

To prove that $v_n\to v_a$ strongly in $E^+$, write
\[
u_n=t_n v_n+w_\delta(v_n),\quad t_n=\frac{\sqrt{a-\norm{w_\delta(v_n)}_2^2}}{\sqrt a}\ge\frac{\sqrt3}{2}.
\]
From $w_\delta(v_n)=u_n^-\to(u_a^\delta)^-$ in $E$, we obtain
\[
t_n\to t_a:=\frac{\sqrt{a-\norm{(u_a^\delta)^-}_2^2}}{\sqrt a}\ge\frac{\sqrt3}{2}.
\]
Consequently,
\[
v_n=\frac{u_n^+}{t_n}\rightarrow v_a:=\frac{(u_a^\delta)^+}{t_a}\quad\text{in }E^+,
\]
which implies $\norm{v_a}_2^2=a$. Moreover,
\[
\norm{v_a}_E^2=\lim_{n\to\infty}\norm{v_n}_E^2\le\frac{3\Theta_a}{4}<\Theta_a,
\]
and hence $v_a\in X_a\subset S_a^+\cap\mathcal C$.

By the continuity of $w_\delta$,
\[
w_\delta(v_a)=\lim_{n\to\infty}w_\delta(v_n)=(u_a^\delta)^-.
\]
It follows from the definition of $t_a$ that
\[
u_a^\delta=t_a v_a+w_\delta(v_a)=G_\delta(v_a).
\]
Hence $J_\delta(v_a)=I_\delta(u_a^\delta)=e_a^\delta$.

Finally, passing to the limit in \eqref{eq:PS-bound}, we obtain
\[
\norm{u_a^\delta}_E\le C\sqrt a.
\]
This proves \eqref{eq:truncated-properties} for some constant $C_0>0$ independent of $a$.
\end{proof}

The conclusions \eqref{eq:truncated-solution} and \eqref{eq:truncated-properties} also hold for every minimizer of $J_\delta$ on $X_a$. Indeed, let $v\in X_a$ satisfy $J_\delta(v)=e_a^\delta$. By Lemma~4.4, $v$ belongs to the relative interior of $X_a$, and therefore
\[
d(J_\delta|_{S_a^+})(v)=0.
\]
Set $u=G_\delta(v)$ and $\mu=\kappa_\delta(u)$.
Then $\mathcal H_\delta(u)=0$ by \eqref{eq:lifting}.
Since $v\in X_a$, we deduce from \eqref{eq:w-estimate} that
$\|u\|_E\le C\sqrt a$, uniformly over these minimizers.
As in the proof of Lemma~\ref{lem:multiplier}, we obtain
$\mu\ge m-Ca^{\frac{q-2}{2}}$.
Testing the equation against $u$, we have
\[
\mu a=2e_a^\delta-\int_{\mathbb R^3}|x|^{-b}\bigl(|u|F_\delta'(|u|)-2F_\delta(|u|)\bigr)\,dx\le 2e_a^\delta<ma.
\]
After decreasing $a_2$ and increasing $C_0$ if necessary,
the bounds $0<\mu<m$ and $\|u\|_E\le C_0\sqrt a$
therefore hold for every such minimizer.
The estimates in Section~\ref{sec:recovery} therefore hold uniformly for $u=G_\delta(v)$ over all minimizers $v$ of $J_\delta$ on $X_a$.

\section{Proof of Theorem~\ref{thm:main}}\label{sec:recovery}
Let $(u_a^\delta,\mu_a^\delta)$ be the normalized solution of the truncated problem obtained in Proposition~\ref{prop:PS}. To complete the proof of Theorem~\ref{thm:main}, we shall prove the uniform estimate
\[
\norm{u_a^\delta}_\infty+\norm{u_a^\delta}_{H^1}
\leq C\sqrt a
\]
for sufficiently small $a>0$, with $C$ independent of $a$. Since the truncated and original nonlinearities coincide on $[0,\delta]$, the $L^\infty$ bound will also ensure that $u_a^\delta$ solves the original equation for sufficiently small $a$.

We start from the bound $\norm{u_a^\delta}_E\le C_0\sqrt a$ in Proposition~\ref{prop:PS}. By \eqref{eq:normequiv} and the Sobolev embeddings $H^{\frac12}(\R^3;\C^4)\hookrightarrow L^r(\R^3;\C^4)$, $2\le r\le3$, we have
\[
\norm{u_a^\delta}_r\le C_r\sqrt a,\quad 2\le r\le3.
\]
We then apply the $L^r$ regularity estimates for the free Dirac operator to the truncated equation and iterate Sobolev embeddings to establish the $L^\infty$ estimate. Combining this estimate with the mass constraint, we estimate the weighted nonlinear term in $L^2(\R^3;\C^4)$. It follows from the equation that $\norm{H_0u_a^\delta}_2\le C\sqrt a$, which implies the required $H^1$ estimate.

\begin{lemma}\label{lem:LrDirac}
Let $1<r<\infty$. Then
\[
H_0^{-1}\in\mathcal L\bigl(L^r(\R^3;\C^4),W^{1,r}(\R^3;\C^4)\bigr),
\]
and there exists $C_r>0$ such that
\begin{equation}\label{eq:H0-inverse-Lr}
\norm{H_0^{-1}f}_{W^{1,r}}\le C_r\norm{f}_r,\quad f\in L^r(\R^3;\C^4).
\end{equation}
\end{lemma}

\begin{proof}
It suffices to establish the estimate for $f$ in the Schwartz class $\mathcal S(\R^3;\C^4)$, since the general case follows by density. By the anticommutation relations, we have
\[
(\alpha\cdot\xi+m\beta)^2=(|\xi|^2+m^2)I_4,
\]
and hence
\[
\widehat{H_0^{-1}f}(\xi)=M_0(\xi)\widehat f(\xi),\quad M_0(\xi)=\frac{\alpha\cdot\xi+m\beta}{m^2+|\xi|^2}.
\]
For every multi-index $\gamma$ with $|\gamma|\le2$,
\[
\norm{\partial_\xi^\gamma M_0(\xi)}_{\mathrm{op}} \le C_\gamma(1+|\xi|)^{-1-|\gamma|}.
\]
Moreover, for $j=1,2,3$,
\[
\widehat{\partial_j H_0^{-1}f}(\xi)=iM_j(\xi)\widehat f(\xi),\quad M_j(\xi)=\frac{\xi_j(\alpha\cdot\xi+m\beta)}{m^2+|\xi|^2},
\]
and the matrices $M_j$ satisfy
\[
\norm{\partial_\xi^\gamma M_j(\xi)}_{\mathrm{op}}\le C_\gamma(1+|\xi|)^{-|\gamma|},\quad|\gamma|\le2.
\]
Applying Mikhlin multiplier theorem~\cite{Grafakos2014}
to the entries of $M_0$ and $M_j$, $j=1,2,3$, we obtain
\[
\norm{H_0^{-1}f}_r+\sum_{j=1}^3\norm{\partial_j H_0^{-1}f}_r\le C_r\norm{f}_r,
\]
which proves \eqref{eq:H0-inverse-Lr}.
\end{proof}

\begin{proposition}\label{prop:Linfty}
For every $M>0$, there exist constants $a_3>0$ and $C_{\delta,M}>0$, depending only on $m,b,q,p,\delta$ and $M$, such that if $(\mu,u)\in(0,m)\times E$ satisfies
\begin{equation}\label{eq:generic-truncated}
H_0u=\mu u+|x|^{-b}f_\delta(|u|)u\quad\text{in }E^*
\end{equation}
and
\[
\norm{u}_2^2=a\in(0,a_3),\quad \norm{u}_E\le M\sqrt a,
\]
then $u\in H^1(\R^3;\C^4)\cap L^\infty(\R^3;\C^4)$ and
\begin{equation}\label{eq:uniform-regularity}
\norm{u}_{H^1}+\norm{u}_\infty\le C_{\delta,M}\sqrt a.
\end{equation}
Consequently, if
\[
a<\min\left\{a_3,\left(\frac{\delta}{C_{\delta,M}}\right)^2\right\},
\]
then $\norm{u}_\infty<\delta$, the truncation is inactive, and
\[
f_\delta(|u|)=|u|^{q-2}-|u|^{p-2}\quad\text{a.e. in }\R^3.
\]
In particular, taking $M=C_0$, the estimate applies to the solutions $u_a^\delta$ obtained in Proposition~\ref{prop:PS} whenever $0<a<\min\{a_2,a_3\}$.
\end{proposition}

\begin{proof}
All constants below may depend on $m,b,q,p,\delta$ and $M$, but are independent of $a$, $\mu$ and $u$. We may assume that $0<a\le1$, so that
$a^{\frac{q-1}{2}}\le\sqrt a$. By \eqref{eq:trunc-growth},
\[
\bigl||x|^{-b}f_\delta(|u|)u\bigr|\le C_\delta|x|^{-b}|u|^{q-1}.
\]
Combining \eqref{eq:normequiv} with the embeddings $H^{\frac12}(\R^3;\C^4)\hookrightarrow L^r(\R^3;\C^4)$, $2\le r\le3$, and the bound $\norm{u}_E\le M\sqrt a$, we deduce that
\[
\norm{u}_r\le C\sqrt a,\quad2\le r\le3.
\]

Choose
\begin{equation}\label{eq:nu-choice}
0<\nu<\min\left\{\frac b6,\,\frac{1-b}{6},\,\frac{q-1}{12},\,\frac{3-b-q}{6}\right\}.
\end{equation}
We now apply a bootstrap argument, starting from $R=3$.
Assume that, for some $R\ge3$,
\[
\norm{u}_2+\norm{u}_R\le C_R\sqrt a.
\]
Define
\[
\frac1{\theta_0}=\frac b3+\nu,\quad\frac1{\theta_\infty}=\frac b3-\nu,\quad\frac1{r_0}=\frac1R+\frac{2\nu}{q-1},
\quad\frac1\tau=\frac{q-1}{R}+\frac b3+\nu.
\]
Then $|x|^{-b}\in L^{\theta_0}(B_1;\R)$ and
$|x|^{-b}\in L^{\theta_\infty}(\R^3\setminus B_1;\R)$.
Moreover, it follows from \eqref{eq:nu-choice} that
\[
\frac1R\le\frac1{r_0}\le\frac13+\frac16=\frac12,
\]
and hence $2\le r_0\le R$. Interpolating between
$L^2(\R^3;\C^4)$ and $L^R(\R^3;\C^4)$, we have
$\norm{u}_{r_0}\le C_R\sqrt a$. Since
\[
\frac1\tau=\frac{q-1}{R}+\frac1{\theta_0}
=\frac{q-1}{r_0}+\frac1{\theta_\infty},
\]
we obtain from H\"older's inequality
\[
\norm{|x|^{-b}|u|^{q-1}}_{L^\tau(B_1)}\le C_R a^{\frac{q-1}{2}},\quad
\norm{|x|^{-b}|u|^{q-1}}_{L^\tau(\R^3\setminus B_1)}\le C_R a^{\frac{q-1}{2}}.
\]
Therefore,
\begin{equation}\label{eq:bootstrap-source}
\norm{|x|^{-b}|u|^{q-1}}_\tau\le C_R a^{\frac{q-1}{2}},\quad\frac1\tau=\frac{q-1}{R}+\frac b3+\nu.
\end{equation}
Indeed,
\[
\frac1\tau\le\frac{q-1+b}{3}+\nu<\frac23+\frac16<1,
\]
where we used $q<3-b$ and \eqref{eq:nu-choice}. Hence $\tau>1$.

Equation~\eqref{eq:generic-truncated} holds in $\mathcal S'(\R^3;\C^4)$. Since the symbol $M_0$ defined in the proof of Lemma~\ref{lem:LrDirac} is smooth and all its derivatives have at most polynomial growth, $H_0^{-1}$ is well defined as a Fourier multiplier on $\mathcal S'(\R^3;\C^4)$. Using $M_0(\xi)(\alpha\cdot\xi+m\beta)=I_4$, we have $H_0^{-1}H_0u=u$ in $\mathcal S'(\R^3;\C^4)$. Moreover,
\[
\mu u\in L^R(\R^3;\C^4)\quad\text{and}\quad|x|^{-b}f_\delta(|u|)u\in L^\tau(\R^3;\C^4).
\]
Applying $H_0^{-1}$ to \eqref{eq:generic-truncated}, we obtain
\[
u=g_1+g_2,\quad g_1=H_0^{-1}(\mu u),\quad g_2=H_0^{-1}\bigl(|x|^{-b}f_\delta(|u|)u\bigr).
\]
By \eqref{eq:trunc-growth}, \eqref{eq:H0-inverse-Lr}, \eqref{eq:bootstrap-source}, and $0<\mu<m$, we have
\begin{equation}\label{eq:split-estimates}
\norm{g_1}_{W^{1,R}}\le C_R\sqrt a,\quad\norm{g_2}_{W^{1,\tau}}\le C_R a^{\frac{q-1}{2}},
\end{equation}
where the constants are uniform for $\mu\in(0,m)$. Suppose that $1<\tau<3$ and set $R'=\tau^*=\frac{3\tau}{3-\tau}$. Then
\[
\frac1{R'}=\frac{q-1}{R}-\frac{1-b}{3}+\nu
\]
and
\begin{equation}\label{eq:uniform-descent}
\frac1{R'}-\frac1R=\frac{q-2}{R}-\frac{1-b}{3}+\nu\le-\left(\frac{3-b-q}{3}-\nu\right)<0.
\end{equation}
Thus $R'>R$. By the embedding
$W^{1,\tau}(\R^3;\C^4)\hookrightarrow L^{R'}(\R^3;\C^4)$,
we have $\norm{g_2}_{R'}\le C_R a^{\frac{q-1}{2}}$.
If $R=3$, we apply the Gagliardo--Nirenberg inequality to obtain
\[
\norm{g_1}_s\le C_s\sqrt a
\quad\text{for every }3\le s<\infty.
\]
If $R>3$, it follows from the embedding $W^{1,R}(\R^3;\C^4)\hookrightarrow L^\infty(\R^3;\C^4)$ and interpolation with the $L^R$ bound in \eqref{eq:split-estimates} that
\[
\norm{g_1}_{R'}\le C_{R'}\sqrt a.
\]
Since $a^{\frac{q-1}{2}}\le\sqrt a$, we deduce that
\[
\norm{u}_{R'}\le C_{R'}\sqrt a.
\]

Starting from $R_0=3$, define
\[
\frac1{\tau_k}=\frac{q-1}{R_k}+\frac b3+\nu,\quad R_{k+1}=\frac{3\tau_k}{3-\tau_k}\quad\text{whenever }\tau_k<3.
\]
By \eqref{eq:nu-choice}, we have $1<\tau_0<3$, so $R_1$ is well defined.
From \eqref{eq:uniform-descent}, we obtain $R_1>R_0=3$ and 
\[
\frac1{R_{k+1}}\le\frac1{R_k}-\left(\frac{3-b-q}{3}-\nu\right)
\]
whenever $\tau_k<3$. Since $\frac1{R_k}\ge0$ and
$\frac{3-b-q}{3}-\nu>0$, there exists an index $k_0\ge1$
such that $\tau_{k_0}\ge3$.
If $\tau_{k_0}>3$, we use Morrey--Sobolev embedding and
\eqref{eq:split-estimates} to obtain
\[
\norm{g_2}_\infty\le C a^{\frac{q-1}{2}},\quad\norm{g_1}_\infty\le C\sqrt a,
\]
and hence $\norm{u}_\infty\le C(\sqrt a+a^{\frac{q-1}{2}})$. If $\tau_{k_0}=3$, we apply Gagliardo--Nirenberg inequality
to $g_2$ and deduce that, for every finite $s\ge3$,
\[
\norm{g_2}_s\le C_s\norm{\nabla g_2}_3^{1-\frac3s}\norm{g_2}_3^{\frac3s}\le C_s a^{\frac{q-1}{2}}.
\]
Since $R_{k_0}>3$, we also have
\[
g_1\in W^{1,R_{k_0}}(\R^3;\C^4)\hookrightarrow L^\infty(\R^3;\C^4).
\]
Interpolating between the $L^3$ and $L^{R_{k_0}}$
bounds for $u$, we obtain $\norm{u}_s\le C_s\sqrt a$
for $3\le s\le R_{k_0}$.
For $s>R_{k_0}$, we have $\norm{g_1}_s\le C_s\sqrt a$
by interpolation between its $L^{R_{k_0}}$ and $L^\infty$ bounds.
Combining this estimate with $\norm{g_2}_s\le C_s a^{\frac{q-1}{2}}$
and using $0<a\le1$, we conclude that
$\norm{u}_s\le C_s\sqrt a$ for every finite $s\ge3$.
Choose $s>R_{k_0}$ sufficiently large so that
\[
\frac{q-1}{s}+\frac b3+\nu<\frac13.
\]
 Taking $R=s$ in \eqref{eq:bootstrap-source}, we have $\tau>3$.
Using \eqref{eq:split-estimates} and Morrey--Sobolev embedding,
we therefore obtain $\norm{u}_\infty\le C(\sqrt a+a^{\frac{q-1}{2}})$.

Since $a^{\frac{q-1}{2}}\le\sqrt a$, we deduce that $\norm{u}_\infty\le C\sqrt a$. The number of bootstrap steps and all constants involved depend only on $m,b,q,p,\delta$, and $M$. We may therefore fix $a_3\in(0,1]$ so that the estimate holds for every $a\in(0,a_3)$.

It remains to estimate $\norm{u}_{H^1}$.
Since $|x|^{-b}\le1$ on $\R^3\setminus B_1$ and
$|x|^{-b}\in L^2(B_1;\R)$ for $b<1$, we obtain from
\eqref{eq:trunc-growth}
\begin{equation}\label{eq:nonlinear-L2-bound}
\norm{|x|^{-b}f_\delta(|u|)u}_2\le C_\delta\left(\norm{|x|^{-b}}_{L^2(B_1)}\norm{u}_\infty^{q-1}
+\norm{u}_\infty^{q-2}\norm{u}_2\right)\le C a^{\frac{q-1}{2}}\le C\sqrt a.
\end{equation}
By \eqref{eq:nonlinear-L2-bound} and $0<\mu<m$, the right-hand side of \eqref{eq:generic-truncated} belongs to $L^2(\R^3;\C^4)$. Hence
\[
\norm{H_0u}_2\le m\norm{u}_2+\norm{|x|^{-b}f_\delta(|u|)u}_2\le C\sqrt a.
\]
Since $u\in L^2(\R^3;\C^4)$ and
$H_0u\in L^2(\R^3;\C^4)$ in the sense of distributions,
it follows from the Fourier representation of $H_0$ that
$u\in H^1(\R^3;\C^4)$.
By Plancherel's theorem and
$(\alpha\cdot\xi+m\beta)^*(\alpha\cdot\xi+m\beta)
=(|\xi|^2+m^2)I_4$, we also have
\[
\norm{H_0u}_2^2=\norm{\nabla u}_2^2+m^2\norm{u}_2^2.
\]
Consequently,
\[
\norm{u}_{H^1}\le C\bigl(\norm{H_0u}_2+\norm{u}_2\bigr)\le C\sqrt a.
\]
Combining this estimate with the $L^\infty$ bound, we obtain
\eqref{eq:uniform-regularity} after increasing the constant
$C_{\delta,M}$ if necessary.
\end{proof}

\begin{proof}[Completion of the proof of Theorem~\ref{thm:main}]
Let $a_3>0$ and $C_{\delta,C_0}>0$ be the constants in
Proposition~\ref{prop:Linfty} for $M=C_0$. Define
\[
a^*:=\min\left\{a_2,\,a_3,\,\left(\frac{\delta}{C_{\delta,C_0}}\right)^2\right\}.
\]
 For $a\in(0,a^*)$, we obtain from Proposition~\ref{prop:PS}
a pair $(\mu_a^\delta,u_a^\delta)\in(0,m)\times E$ such that
\[
\norm{u_a^\delta}_2^2=a,\quad I_\delta(u_a^\delta)=e_a^\delta,\quad\norm{u_a^\delta}_E\le C_0\sqrt a.
\]
By Proposition~\ref{prop:Linfty}, we have
\[
\norm{u_a^\delta}_\infty\le C_{\delta,C_0}\sqrt a<\delta,
\]
and therefore
\[
f_\delta(|u_a^\delta|)=|u_a^\delta|^{q-2}-|u_a^\delta|^{p-2},\quad
F_\delta(|u_a^\delta|)=F(|u_a^\delta|)\quad\text{a.e. in }\R^3.
\]
Applying Lemma~\ref{lem:original-regular-branch}, we deduce that $u_a^\delta\in\mathcal D_p$ and that the nonlinear term in the original equation belongs to $L^2(\R^3;\C^4)$. Thus $(\mu_a^\delta,u_a^\delta)$ satisfies \eqref{eq:main}. Since $F_\delta(|u_a^\delta|)=F(|u_a^\delta|)$ a.e. in $\R^3$, we also have
\[
I(u_a^\delta)=I_\delta(u_a^\delta)=e_a^\delta.
\]
By \eqref{eq:truncated-properties},
\[
0<I(u_a^\delta)<\frac{ma}{2}.
\]
Using \eqref{eq:uniform-regularity}, we obtain
\[
\norm{u_a^\delta}_{H^1}+\norm{u_a^\delta}_\infty\le C\sqrt a
\]
with $C$ independent of $a$. Since $\delta$ was fixed in terms of $p$ and $q$, all thresholds and constants above depend only on $m,b,q$, and $p$. The conclusion follows by setting $(u_a,\mu_a)=(u_a^\delta,\mu_a^\delta)$.
\end{proof}

\section{Proof of Theorem~\ref{thm:small-mass}}\label{sec:small-mass}
Throughout this section, we fix $m,b,q,p,\delta$ and study the small mass behavior of an arbitrary family of reduced minimizers. Using interpolation estimates and the properties of the limiting Schr\"odinger problem, we estimate the reduced energy on the scale $a\varepsilon_a^2$ and prove compactness after rescaling. We then derive the asymptotic formulas for the energy, the frequency and the lower component of the original spinor.

By Proposition~\ref{prop:PS}, $J_\delta$ admits a minimizer on $X_a$ for every sufficiently small $a>0$. For each such $a$, choose any minimizer $v_a$ at the level $e_a^\delta$, and let $(u_a,\mu_a)$ be the corresponding solution obtained in Section~\ref{sec:recovery}. Then
\begin{equation}\label{eq:limit-decomposition}
\begin{aligned}
u_a&=G_\delta(v_a)=t_av_a+w_a,\\
w_a&=w_\delta(v_a)\in E^-,
\quad
t_a=\left(1-\frac{\norm{w_a}_2^2}{a}\right)^{\frac{1}{2}}.
\end{aligned}
\end{equation}
Since the truncation is inactive for $u_a$ when $a$ is sufficiently
small, we have
\[
I(u_a)=I_\delta(u_a)=J_\delta(v_a)=e_a^\delta.
\]
As $v_a\in E^+$, there exists a unique
$\omega_a\in H^{\frac{1}{2}}(\R^3;\C^2)$ such that
\begin{equation}\label{eq:limit-FW-coordinate}
v_a=U_{\rm FW}^{-1}\binom{\omega_a}{0},
\quad
\norm{\omega_a}_2^2=\norm{v_a}_2^2=a.
\end{equation}
With $\varepsilon_a$ as defined in the introduction, set
\begin{equation}\label{eq:normalized-dilation}
(\mathscr D_a z)(x)=a^{-\frac{1}{2}}\varepsilon_a^{-\frac{3}{2}}z\left(\frac{x}{\varepsilon_a}\right).
\end{equation}
Then $\norm{\mathscr D_a z}_2^2=a^{-1}\norm{z}_2^2$. In particular, the rescaled solution in
\eqref{eq:main-rescaling} satisfies
\[
\widetilde u_a=\mathscr D_a u_a,\quad\norm{\widetilde u_a}_2^2=1.
\]

\subsection{The limiting Schr\"{o}dinger problem}
For $2<r\le3-b$, set
\[
\gamma_r=\frac{3(r-2)}2+b.
\]
Note that $0<\gamma_r<\frac{r}{2}$ when $r<3-b$, while $\gamma_{3-b}=\frac{3-b}{2}$. We also fix the auxiliary exponent
\begin{equation}\label{eq:r-star}
r_*=q+\frac12\min\{p-q,3-b-q\}.
\end{equation}
It satisfies
\begin{equation}\label{eq:r-star-range}
q<r_*<\min\{p,3-b\}.
\end{equation}
We shall use the following estimates.

\begin{lemma}\label{lem:limit-interpolation}
Let $2<r\le3-b$ and $z\in H^{\frac12}(\R^3;\C^N)$, where $N\in\{2,4\}$. Then
\begin{equation}\label{eq:weighted-GN-half}
\mathcal A_r(z)\le C_r\norm{z}_2^{r-2\gamma_r}\left(\int_{\R^3}|\xi||\widehat z(\xi)|^2\,d\xi\right)^{\gamma_r}.
\end{equation}
Let $\omega\in H^{\frac12}(\R^3;\C^2)$ satisfy $\norm{\omega}_2^2=a$. Set
\[
v=U_{\rm FW}^{-1}\binom{\omega}{0},\quad K(\omega):=\int_{\R^3}(\lambda(\xi)-m)|\widehat\omega(\xi)|^2\,d\xi.
\]
Then $K(\omega)<\infty$. Moreover, we have
\begin{equation}\label{eq:FW-D-control}
\left(\int_{\R^3}|\xi||\widehat\omega(\xi)|^2\,d\xi\right)^2\le K(\omega)\bigl(K(\omega)+2ma\bigr),
\end{equation}
and
\begin{equation}\label{eq:omega-weighted-K}
\mathcal A_r(\omega)\le C_r a^{\frac r2-\gamma_r}\bigl[K(\omega)\bigl(K(\omega)+2ma\bigr)\bigr]^{\frac{\gamma_r}{2}}.
\end{equation}
In addition,
\begin{equation}\label{eq:FW-close-upper}
\left\|v-\binom{\omega}{0}\right\|_E^2\le C K(\omega),
\end{equation}
and
\begin{equation}\label{eq:FW-Ar-comparison-sharp}
\begin{split}
\mathcal A_r(v)&\le C_r\bigl(\mathcal A_r(\omega)+K(\omega)^{\frac r2}\bigr),\\
\left|\mathcal A_r(v)-\mathcal A_r(\omega)\right|&\le C_r\left(\mathcal A_r(\omega)^{\frac{r-1}{r}}
+K(\omega)^{\frac{r-1}{2}}\right)K(\omega)^{\frac12}.
\end{split}
\end{equation}
\end{lemma}

\begin{proof}
Put $\ell_r=\frac{r-2b}{1-b}$. Since $2<r\le3-b$ and $0<b<1$, we have $2<\ell_r\le3$. Applying H\"{o}lder's inequality, we deduce that
\[
\mathcal A_r(z)\le\left(\int_{\R^3}|x|^{-1}|z|^2\,dx\right)^b\left(\int_{\R^3}|z|^{\ell_r}\,dx\right)^{1-b}.
\]
By fractional Hardy inequality, the embedding
$\dot H^{\frac12}(\R^3;\C^N)\hookrightarrow L^3(\R^3;\C^N)$,
and interpolation between $L^2(\R^3;\C^N)$ and
$L^3(\R^3;\C^N)$, we obtain
\[
\int_{\R^3}|x|^{-1}|z|^2\,dx\le C\norm{z}_{\dot H^{\frac12}}^2,\quad
\norm{z}_{\ell_r}^{\ell_r}\le C\norm{z}_2^{6-2\ell_r}\norm{z}_{\dot H^{\frac12}}^{3(\ell_r-2)}.
\]
Since
\[
(1-b)(6-2\ell_r)=r-2\gamma_r,\quad b+\frac32(1-b)(\ell_r-2)=\gamma_r,
\]
we deduce \eqref{eq:weighted-GN-half}.

As $0\le\lambda(\xi)-m\le|\xi|$ and $\omega\in H^{\frac12}(\R^3;\C^2)$, we have $K(\omega)<\infty$. Using $|\xi|^2=(\lambda(\xi)-m)(\lambda(\xi)+m)$
and Cauchy--Schwarz inequality, we obtain
\[
\begin{aligned}
\int_{\R^3}|\xi||\widehat\omega|^2\,d\xi
&\le\left(\int_{\R^3}(\lambda-m)|\widehat\omega|^2\,d\xi\right)^{1/2}
\left(\int_{\R^3}(\lambda+m)|\widehat\omega|^2\,d\xi\right)^{1/2}\\
&=\bigl[K(\omega)(K(\omega)+2ma)\bigr]^{1/2},
\end{aligned}
\]
which proves \eqref{eq:FW-D-control}. Combining this estimate with \eqref{eq:weighted-GN-half},
we deduce \eqref{eq:omega-weighted-K}.

From Foldy--Wouthuysen representation in the proof of Lemma~\ref{lem:FW-low}, we have
\[
\widehat v(\xi)=\binom{\rho_+(\xi)\widehat\omega(\xi)}{\rho_-(\xi)S(\xi)\widehat\omega(\xi)}.
\]
Since $\lambda\rho_-^2=\frac{\lambda-m}{2}$ and $(1-\rho_+)^2\le\rho_-^2$, it follows that
\[
\left\|v-\binom{\omega}{0}\right\|_E^2\le C\int_{\R^3}(\lambda-m)|\widehat\omega|^2\,d\xi,
\]
which is \eqref{eq:FW-close-upper}.  By Lemma~\ref{lem:weighted}, we have
\[
\left\|v-\binom{\omega}{0}\right\|_{L_b^r}\le C K(\omega)^{\frac12}.
\]
Set $d=v-\binom{\omega}{0}$. Applying the triangle inequality, we obtain
\[
\mathcal A_r(v)^{\frac1r}\le \mathcal A_r(\omega)^{\frac1r}+\norm{d}_{L_b^r}
\le\mathcal A_r(\omega)^{\frac1r}+CK(\omega)^{\frac12},
\]
which proves the first estimate in \eqref{eq:FW-Ar-comparison-sharp}. Moreover, using the pointwise inequality
\[
\bigl||z|^r-|y|^r\bigr|\le C_r\bigl(|z|^{r-1}+|y|^{r-1}\bigr)|z-y|
\]
and H\"{o}lder's inequality with respect to $|x|^{-b}\,dx$, we deduce
\[
\left|\mathcal A_r(v)-\mathcal A_r(\omega)\right|\le C_r\left(\mathcal A_r(v)^{\frac{r-1}{r}}
+\mathcal A_r(\omega)^{\frac{r-1}{r}}\right)\norm{d}_{L_b^r}.
\]
The second estimate in \eqref{eq:FW-Ar-comparison-sharp} now follows from the first one and the bound for $\norm{d}_{L_b^r}$.
\end{proof}

We next consider the minimization problem defining $e_q$.

\begin{lemma}\label{lem:NLS-ground-set} The level $e_q$ defined in \eqref{eq:NLS-ground-set-introduction} satisfies $-\infty<e_q<0$, and the set of minimizers $\mathcal{G}_q$ is nonempty and compact in $H^1(\R^3;\C^2)$. Moreover, every $Q\in\mathcal{G}_q$ is a weak solution of \eqref{eq:NLS-limit-introduction}, with $\Lambda_q$ given by \eqref{eq:Lambda-limit-introduction}.
\end{lemma}

\begin{proof} For $\mathscr E_q$ defined in \eqref{eq:NLS-energy-introduction}, combining \eqref{eq:weighted-GN-half} with
\[
\norm{Q}_{\dot H^{\frac12}}^2\le\norm{Q}_2\norm{\nabla Q}_2,
\]
we have the weighted Gagliardo--Nirenberg estimate
\begin{equation}\label{eq:weighted-GN-one}
\mathcal A_q(Q)\le C\norm{Q}_2^{q-\gamma_q}\norm{\nabla Q}_2^{\gamma_q}.
\end{equation}
Since $0<\gamma_q<2$, it follows from Young's inequality
that $\mathscr E_q$ is bounded below on $\mathcal S_q$. If $Q\in\mathcal S_q$ and $Q_t(x)=t^{\frac32}Q(tx)$, then
\[
\mathscr E_q(Q_t)=\frac{t^2}{4m}\norm{\nabla Q}_2^2-\frac{t^{\gamma_q}}q\mathcal A_q(Q).
\]
For sufficiently small $t>0$, we have $\mathscr E_q(Q_t)<0$, and hence $e_q<0$.

Let $\{Q_n\}\subset\mathcal S_q$ be a minimizing sequence. By \eqref{eq:weighted-GN-one}, $\{Q_n\}$ is bounded in $H^1(\R^3;\C^2)$. After passing to a subsequence, we may assume that
\[
Q_n\rightharpoonup Q\quad\text{in }H^1(\R^3;\C^2),
\quad Q_n\to Q\quad\text{in }L^q(\R^3,|x|^{-b}\,dx;\C^2),
\]
where the weighted convergence follows from Lemma~\ref{lem:weighted}. Using weak lower semicontinuity and $\mathcal A_q(Q_n)\to\mathcal A_q(Q)$, we obtain
\[
\mathscr E_q(Q)\le\liminf_{n\to\infty}\mathscr E_q(Q_n)=e_q<0.
\]
Thus $Q\ne0$. By weak lower semicontinuity of the $L^2$ norm, we have $\theta=\norm{Q}_2^2\in(0,1]$.
If $\theta<1$, set $Q^\sharp=\frac{Q}{\sqrt\theta}$.
Since $q>2$, we deduce that
\[
\mathscr E_q(Q)=\theta\mathscr E_q(Q^\sharp)+\frac{\theta-\theta^{\frac q2}}q\mathcal A_q(Q^\sharp)
\ge\theta e_q+\frac{\theta-\theta^{\frac q2}}q\mathcal A_q(Q^\sharp)>\theta e_q>e_q,
\]
which contradicts $\mathscr E_q(Q)\le e_q$. Hence $\theta=1$. By the definition of $e_q$, we now have $e_q\le\mathscr E_q(Q)$, and therefore $Q\in\mathcal G_q$.

Since $Q_n\rightharpoonup Q$ in $L^2(\R^3;\C^2)$ and $\norm{Q_n}_2=\norm{Q}_2=1$, we obtain
\[
Q_n\to Q\quad\text{strongly in }L^2(\R^3;\C^2).
\]
Moreover, using
$\mathscr E_q(Q_n)\to e_q=\mathscr E_q(Q)$ and
$\mathcal A_q(Q_n)\to\mathcal A_q(Q)$, we have
\[
\begin{aligned}
\frac{1}{4m}\norm{\nabla Q_n}_2^2
&=
\mathscr E_q(Q_n)+\frac1q\mathcal A_q(Q_n)\\
&\rightarrow
\mathscr E_q(Q)+\frac1q\mathcal A_q(Q)
=
\frac{1}{4m}\norm{\nabla Q}_2^2.
\end{aligned}
\]
Hence
\[
Q_n\to Q
\quad\text{strongly in }H^1(\R^3;\C^2).
\]
Since every sequence in $\mathcal G_q$ is a minimizing
sequence for $e_q$, we conclude that
$\mathcal G_q$ is compact in $H^1(\R^3;\C^2)$.

Let $Q\in\mathcal G_q$. By the Lagrange multiplier rule, there exists $\Lambda\in\R$ such that
\[
-\frac1{2m}\Delta Q+\Lambda Q=|x|^{-b}|Q|^{q-2}Q.
\]
As $Q_t\in\mathcal S_q$ for $t>0$ and $Q$ is a minimizer, the function $t\mapsto\mathscr E_q(Q_t)$ has a minimum at $t=1$. Differentiating at $t=1$, we obtain
\begin{equation}\label{eq:NLS-Pohozaev}
\frac1{2m}\norm{\nabla Q}_2^2=\frac{\gamma_q}{q}\mathcal A_q(Q).
\end{equation}
Testing the equation against $Q$ and using $\norm{Q}_2=1$, we also have
\[
\Lambda=\mathcal A_q(Q)-\frac1{2m}\norm{\nabla Q}_2^2.
\]
Combining these two identities with $\mathscr E_q(Q)=e_q$, we deduce that
\[
e_q=-\frac{2-\gamma_q}{2q}\mathcal A_q(Q),\quad
\Lambda=\mathcal A_q(Q)-\frac1{2m}\norm{\nabla Q}_2^2=-\frac{2(q-\gamma_q)}{2-\gamma_q}e_q.
\]
Thus $\Lambda=\Lambda_q$ is independent of the choice of $Q\in\mathcal G_q$ and is positive since $e_q<0$ and $\gamma_q<2<q$.
\end{proof}

\subsection{Estimates for the reduced energy}
Let $r_*$ be the exponent in \eqref{eq:r-star}, satisfying
\eqref{eq:r-star-range}. We estimate the error in replacing
$F_\delta(\sigma)$ with $\sigma^q/q$ in the energy functional.

\begin{lemma}\label{lem:truncated-remainder}
There exists a constant $C_{\delta,r_*}>0$ such that
\begin{equation}\label{eq:truncated-remainder}
\left|F_\delta(\sigma)-\frac{\sigma^q}{q}\right|\le C_{\delta,r_*}\sigma^{r_*},\quad \sigma\ge0.
\end{equation}
Consequently, for every $u\in E$,
\begin{equation}\label{eq:truncated-remainder-integral}
\left|\int_{\R^3}|x|^{-b}F_\delta(|u|)\,dx-\frac1q\mathcal A_q(u)\right|\le C_{\delta,r_*}\mathcal A_{r_*}(u).
\end{equation}
\end{lemma}

\begin{proof}
If $0\le\sigma\le\delta$, then
\[
\left|F_\delta(\sigma)-\frac{\sigma^q}{q}\right|=\frac{\sigma^p}{p}\le\frac{\delta^{p-r_*}}p\sigma^{r_*}.
\]
For $\sigma\ge\delta$, it follows from
\eqref{eq:trunc-definition} that
\[
F_\delta(\sigma)-\frac{\sigma^q}{q}=\left(A_\delta-\frac1q\right)\sigma^q+B_\delta\sigma^2+\Gamma_\delta.
\]
Since $r_*>q>2$ and $\sigma\ge\delta$, we have
$\sigma^q+\sigma^2+1\le C_{\delta,r_*}\sigma^{r_*}$, which proves \eqref{eq:truncated-remainder}. Taking $\sigma=|u(x)|$ in \eqref{eq:truncated-remainder} and integrating with respect to $|x|^{-b}\,dx$, we obtain \eqref{eq:truncated-remainder-integral}. Its right-hand side is finite by Lemma~\ref{lem:weighted} and $r_*<3-b$.
\end{proof}

We next estimate the energy increase resulting from maximization over the negative spectral fiber.

\begin{lemma}\label{lem:fiber-energy-error}
There exists a constant $C>0$ such that, for every $v\in S_a^+\cap\mathcal C$ with $0<a<a_0$,
\begin{equation}\label{eq:fiber-energy-error}
0\le J_\delta(v)-I_\delta(v)\le C\mathcal A_q(v)^{\frac{2(q-1)}q}.
\end{equation}
\end{lemma}

\begin{proof}
With the notation in the proof of Lemma~\ref{lem:fiber}, the point $w_\delta(v)$ maximizes $\Phi_v$, so the left-hand inequality is immediate. By \eqref{eq:concave},
\[
\Phi_v(w)\le\Phi_v(0)+D_w\Phi_v(0)[w]-\frac14\norm{w}_E^2,\quad w\in\mathcal B(v).
\]
As in the proof of Lemma~\ref{lem:w-estimate}, we have
\[
\norm{D_w\Phi_v(0)}_{(E^-)^*}\le C\mathcal A_q(v)^{\frac{q-1}{q}}.
\]
Set $\eta_v=C\mathcal A_q(v)^{\frac{q-1}{q}}$. Combining the above two estimates, we obtain
\[
0\le J_\delta(v)-I_\delta(v)\le\sup_{s\ge0}\left(\eta_vs-\frac14s^2\right)=\eta_v^2.
\]
This is \eqref{eq:fiber-energy-error}.
\end{proof}

By the definition of $\varepsilon_a$, we have
\begin{equation}\label{eq:epsilon-balance}
a^{\frac{q-2}{2}}\varepsilon_a^{\gamma_q}=\varepsilon_a^2.
\end{equation}

\begin{lemma}\label{lem:sharp-energy}
Let $v_a$, $w_a$, $t_a$, and $\omega_a$ be given by \eqref{eq:limit-decomposition}--\eqref{eq:limit-FW-coordinate}, and set
\begin{equation*}\label{eq:Ka-definition}
K_a=\norm{v_a}_E^2-ma=\int_{\R^3}(\lambda(\xi)-m)|\widehat\omega_a(\xi)|^2\,d\xi.
\end{equation*}
Then
\begin{equation}\label{eq:sharp-bounds}
K_a=O(a\varepsilon_a^2),\quad\mathcal A_q(v_a)+\mathcal A_q(\omega_a)=O(a\varepsilon_a^2),
\quad\mathcal A_{r_*}(v_a)=o(a\varepsilon_a^2),
\end{equation}
and
\begin{equation}\label{eq:sharp-w-bounds}
\norm{w_a}_E^2=o(a\varepsilon_a^2),\quad t_a=1+o(\varepsilon_a^2).
\end{equation}
Moreover,
\begin{equation}\label{eq:sharp-energy-reduction}
e_a^\delta=\frac{ma}{2}+\frac{K_a}{2}-\frac1q\mathcal A_q(\omega_a)+o(a\varepsilon_a^2).
\end{equation}
\end{lemma}

\begin{proof}
To estimate $e_a^\delta$ from above, fix $\phi\in C_c^\infty(\R^3;\C^2)$ with $\norm{\phi}_2=1$ and define
\[
\omega_a^\phi(x)=\sqrt a\,\varepsilon_a^{\frac32}\phi(\varepsilon_a x),\quad
v_a^\phi=U_{\rm FW}^{-1}\binom{\omega_a^\phi}{0}.
\]
Then $v_a^\phi\in X_a$ for all sufficiently small $a$. Applying the dominated convergence theorem, we obtain
\begin{equation}\label{eq:test-sharp-kinetic}
\frac{\norm{v_a^\phi}_E^2-ma}{a\varepsilon_a^2}=\int_{\R^3}\frac{\sqrt{m^2+\varepsilon_a^2|\xi|^2}-m}{\varepsilon_a^2}
|\widehat\phi(\xi)|^2\,d\xi\rightarrow\frac1{2m}\norm{\nabla\phi}_2^2.
\end{equation}
Indeed,
\[
0\le\frac{\sqrt{m^2+\varepsilon_a^2|\xi|^2}-m}{\varepsilon_a^2}=\frac{|\xi|^2}{\sqrt{m^2+\varepsilon_a^2|\xi|^2}+m}
\le\frac{|\xi|^2}{2m}.
\]
By Fourier scaling and the definition of $U_{\rm FW}$, we have
\begin{equation}\label{eq:test-rescaled-FW}
\widehat{\mathscr D_a v_a^\phi}(\xi)=U(\varepsilon_a\xi)^{-1}\binom{\widehat\phi(\xi)}{0}.
\end{equation}
For every $\xi\in\R^3$, $U(\varepsilon_a\xi)^{-1}\to I_4$
as $a\to0^+$. Since $U(\varepsilon_a\xi)^{-1}$ is unitary,
\[
\left(1+|\xi|^2\right)^{\frac12}\left|\left[U(\varepsilon_a\xi)^{-1}-I_4\right]
\binom{\widehat\phi(\xi)}{0}\right|^2\le4\left(1+|\xi|^2\right)^{\frac12}|\widehat\phi(\xi)|^2.
\]
The right-hand side is integrable because
$\phi\in C_c^\infty(\R^3;\C^2)$.
It follows from \eqref{eq:test-rescaled-FW} and the dominated convergence theorem that
\[
\mathscr D_a v_a^\phi\rightarrow\binom{\phi}{0}\quad\text{in }H^{\frac12}(\R^3;\C^4).
\]
Applying Lemma~\ref{lem:weighted}, we also have
\[
\mathscr D_a v_a^\phi\rightarrow\binom{\phi}{0}\quad\text{in }L^{r}(\R^3,|x|^{-b}\,dx;\C^4),\quad r\in\{q,r_*\}.
\]
Using the scaling identity
\[
\mathcal A_r(v_a^\phi)=a^{\frac{r}{2}}\varepsilon_a^{\gamma_r}\mathcal A_r(\mathscr D_a
v_a^\phi),
\]
we deduce that
\begin{equation}\label{eq:test-sharp-nonlinear}
\mathcal A_r(v_a^\phi)=a^{\frac{r}{2}}\varepsilon_a^{\gamma_r}\bigl(\mathcal A_r(\phi)+o(1)\bigr).
\end{equation}
By \eqref{eq:epsilon-balance} and \eqref{eq:test-sharp-nonlinear},
\[
\mathcal A_q(v_a^\phi)=a\varepsilon_a^2\bigl(\mathcal A_q(\phi)+o(1)\bigr),\quad
\mathcal A_{r_*}(v_a^\phi)=o(a\varepsilon_a^2).
\]
The latter estimate follows from
\[
\frac{a^{\frac{r_*}{2}}\varepsilon_a^{\gamma_{r_*}}}{a\varepsilon_a^2}
=a^{\frac{r_*-q}{2}}\varepsilon_a^{\gamma_{r_*}-\gamma_q}
=a^{\frac{r_*-q}{2}}\varepsilon_a^{\frac{3(r_*-q)}2}\rightarrow0.
\]
Moreover, combining \eqref{eq:fiber-energy-error} with
$\mathcal A_q(v_a^\phi)=O(a\varepsilon_a^2)$, we obtain
\[
J_\delta(v_a^\phi)-I_\delta(v_a^\phi)=O\left((a\varepsilon_a^2)^{\frac{2(q-1)}q}\right)=o(a\varepsilon_a^2).
\]
Therefore, by \eqref{eq:test-sharp-kinetic} and Lemma~\ref{lem:truncated-remainder},
\[
J_\delta(v_a^\phi)-\frac{ma}{2}=a\varepsilon_a^2\mathscr E_q(\phi)+o(a\varepsilon_a^2).
\]
Hence
\[
\limsup_{a\to0}\frac{e_a^\delta-\frac{ma}{2}}{a\varepsilon_a^2}
\le\lim_{a\to0}\frac{J_\delta(v_a^\phi)-\frac{ma}{2}}{a\varepsilon_a^2}=\mathscr E_q(\phi).
\]
Since normalized functions in $C_c^\infty(\R^3;\C^2)$
are dense in $\mathcal S_q$ with respect to the $H^1$ norm,
we take the infimum over such $\phi$ to obtain
\begin{equation}\label{eq:energy-limsup}
\limsup_{a\to0}\frac{e_a^\delta-\frac{ma}{2}}{a\varepsilon_a^2}\le e_q<0.
\end{equation}

We now estimate $K_a$.  Using $J_\delta(v_a)\ge I_\delta(v_a)$ and $F_\delta(\sigma)\le C\sigma^q$, we have
\begin{equation}\label{eq:K-lower-energy}
e_a^\delta-\frac{ma}{2}\ge\frac{K_a}{2}-C\mathcal A_q(v_a).
\end{equation}
By Lemma~\ref{lem:limit-interpolation},
\begin{equation}\label{eq:Aq-K-bound}
\mathcal A_q(v_a)\le C a^{\frac q2-\gamma_q}\bigl[K_a(K_a+2ma)\bigr]^{\frac{\gamma_q}{2}}+C
K_a^{\frac{q}{2}}.
\end{equation}
Put $k_a=\frac{K_a}{a}$.
From \eqref{eq:K-lower-energy} and Lemma~\ref{lem:energy-level},
we have $K_a\le C\mathcal A_q(v_a)$ for all sufficiently small $a$.
Combining this bound with \eqref{eq:Aq-K-bound}, we obtain
\[
k_a\le C a^{\frac{q-2}{2}}\left([k_a(k_a+2m)]^{\frac{\gamma_q}{2}}+k_a^{\frac{q}{2}}\right).
\]
Moreover, $v_a\in X_a$, so
\[
0\le k_a\le\frac{3}{2}(m+1)-m=\frac{m+3}{2}.
\]
It follows from these two inequalities and $q>2$ that $k_a\to0^+$.
In particular, $k_a\le1$ for all sufficiently small $a$.
Using $q>\gamma_q$, we have
\[
[k_a(k_a+2m)]^{\frac{\gamma_q}{2}}+k_a^{\frac q2}\le C k_a^{\frac{\gamma_q}{2}}.
\]
Thus $k_a^{1-\frac{\gamma_q}{2}}\le Ca^{\frac{q-2}{2}}$.
Since $\gamma_q<2$, it follows that
\[
k_a\le Ca^{\frac{q-2}{2-\gamma_q}}=C\varepsilon_a^2,
\]
where the last equality follows from \eqref{eq:epsilon-balance}. This proves the bound for $K_a$ in \eqref{eq:sharp-bounds}.
Substituting this estimate into \eqref{eq:Aq-K-bound} and using \eqref{eq:epsilon-balance}, we obtain
\[
\mathcal A_q(v_a)\le C a^{\frac q2}\varepsilon_a^{\gamma_q}+C(a\varepsilon_a^2)^{\frac q2}\le C a\varepsilon_a^2.
\]
Similarly, by \eqref{eq:omega-weighted-K} with $r=q$,
$\mathcal A_q(\omega_a)\le C a\varepsilon_a^2$, and hence
$\mathcal A_q(v_a)+\mathcal A_q(\omega_a)=O(a\varepsilon_a^2)$.
Substituting these estimates and $K_a=O(a\varepsilon_a^2)$
into \eqref{eq:FW-Ar-comparison-sharp}, we deduce
\[
\left|\mathcal A_q(v_a)-\mathcal A_q(\omega_a)\right|=o(a\varepsilon_a^2).
\]
Indeed, the two powers of $a\varepsilon_a^2$ arising on the right-hand side are $\frac{q-1}{q}+\frac12>1$ and $\frac{q}{2}>1$.
Taking $r=r_*$ in \eqref{eq:omega-weighted-K} and
\eqref{eq:FW-Ar-comparison-sharp}, and using
$K_a=O(a\varepsilon_a^2)$, we have
\[
\mathcal A_{r_*}(v_a)\le Ca^{\frac{r_*}{2}}\varepsilon_a^{\gamma_{r_*}}+C(a\varepsilon_a^2)^{\frac{r_*}{2}}
=o(a\varepsilon_a^2),
\]
where the first term is $o(a\varepsilon_a^2)$ by
\eqref{eq:epsilon-balance} and $r_*>q$, and the second is
$o(a\varepsilon_a^2)$ since $r_*>2$.

By Lemma~\ref{lem:w-estimate}, we have
\[
\|w_a\|_E^2\le C\mathcal A_q(v_a)^{\frac{2(q-1)}q}=O\left((a\varepsilon_a^2)^{\frac{2(q-1)}q}\right)=o(a\varepsilon_a^2).
\]
It follows from the definition of $t_a$ and \eqref{eq:normequiv} that
\[
0\le1-t_a\le C\frac{\|w_a\|_2^2}{a}\le C\frac{\|w_a\|_E^2}{a}=o(\varepsilon_a^2),
\]
which proves \eqref{eq:sharp-w-bounds}. Combining \eqref{eq:fiber-energy-error},
\eqref{eq:truncated-remainder-integral}, \eqref{eq:sharp-bounds}, and
$\mathcal A_q(v_a)-\mathcal A_q(\omega_a)=o(a\varepsilon_a^2)$,
we obtain
\[
\begin{aligned}
e_a^\delta
&=I_\delta(v_a)+o(a\varepsilon_a^2)\\
&=\frac{ma}{2}+\frac{K_a}{2}
-\frac1q\mathcal A_q(v_a)+o(a\varepsilon_a^2)\\
&=\frac{ma}{2}+\frac{K_a}{2}
-\frac1q\mathcal A_q(\omega_a)+o(a\varepsilon_a^2),
\end{aligned}
\]
which is \eqref{eq:sharp-energy-reduction}.
\end{proof}

\subsection{Compactness and completion of the proof}
We next establish compactness for the rescaled family
\begin{equation}\label{eq:varpi-rescaling}
\varpi_a(x)
=a^{-\frac12}\varepsilon_a^{-\frac32}
\omega_a\left(\frac{x}{\varepsilon_a}\right),
\quad \norm{\varpi_a}_2^2=1.
\end{equation}
To express the kinetic energy after rescaling, we introduce,
for $\varepsilon>0$,
\begin{equation}\label{eq:pseudo-kinetic}
\mathfrak t_\varepsilon(\xi)
=\frac{\sqrt{m^2+\varepsilon^2|\xi|^2}-m}{\varepsilon^2},
\quad
\mathcal T_\varepsilon(\varpi)
=\int_{\R^3}
\mathfrak t_\varepsilon(\xi)
|\widehat\varpi(\xi)|^2\,d\xi.
\end{equation}
The Fourier transform of $\varpi_a$ is
\[
\widehat\varpi_a(\xi)=a^{-\frac12}\varepsilon_a^{\frac32}\widehat\omega_a(\varepsilon_a\xi).
\]
Then
\[
K_a=a\int_{\R^3}\bigl(\sqrt{m^2+\varepsilon_a^2|\xi|^2}-m\bigr)|\widehat\varpi_a(\xi)|^2\,d\xi
=a\varepsilon_a^2\mathcal T_{\varepsilon_a}(\varpi_a),
\]
\[
\mathcal A_q(\omega_a)=a^{\frac q2}\varepsilon_a^{\gamma_q}\mathcal A_q(\varpi_a)=a\varepsilon_a^2\mathcal A_q(\varpi_a).
\]
The last equality follows from \eqref{eq:epsilon-balance}. Substituting these identities into
\eqref{eq:sharp-energy-reduction}, we have
\begin{equation}\label{eq:rescaled-energy-identity}
\frac{e_a^\delta-\frac{ma}{2}}{a\varepsilon_a^2}=\frac12\mathcal T_{\varepsilon_a}(\varpi_a)
-\frac1q\mathcal A_q(\varpi_a)+o(1).
\end{equation}

\begin{lemma}\label{lem:rescaled-compactness}
As $a\to0^+$, we have
\begin{equation}\label{eq:rescaled-energy-limit}
\frac{e_a^\delta-\frac{ma}{2}}{a\varepsilon_a^2}\rightarrow e_q,\quad
\operatorname{dist}_{H^{\frac12}(\R^3;\C^2)}(\varpi_a,\mathcal G_q)\rightarrow0.
\end{equation}
Moreover, if $a_n\to0^+$ and $\varpi_{a_n}\to Q$ strongly in $H^{\frac12}(\R^3;\C^2)$ for some $Q\in\mathcal G_q$, then
\begin{equation}\label{eq:pseudo-kinetic-limit}
\mathcal T_{\varepsilon_{a_n}}(\varpi_{a_n})\rightarrow\frac1{2m}\norm{\nabla Q}_2^2.
\end{equation}
\end{lemma}

\begin{proof}
From \eqref{eq:sharp-bounds} and the identity $K_a=a\varepsilon_a^2\mathcal T_{\varepsilon_a}(\varpi_a)$, we have 
$\mathcal T_{\varepsilon_a}(\varpi_a)=O(1)$. Applying \eqref{eq:FW-D-control} to $\omega_a$ and using \eqref{eq:varpi-rescaling}, we obtain
\[
\left(\int_{\R^3}|\xi|\,|\widehat\varpi_a|^2\,d\xi\right)^2\le\mathcal T_{\varepsilon_a}(\varpi_a)
\left(\varepsilon_a^2\mathcal T_{\varepsilon_a}(\varpi_a)+2m\right).
\]
Together with $\norm{\varpi_a}_2=1$, this proves that $\{\varpi_a\}$ is bounded in $H^{\frac12}(\R^3;\C^2)$.

Let $a_n\to0^+$ be arbitrary. After passing to a subsequence, we may assume that

\[
\varpi_{a_n}\rightharpoonup Q
\quad\text{in }H^{\frac12}(\R^3;\C^2),
\quad\varpi_{a_n}\rightarrow Q
\quad\text{in }L^q(\R^3,|x|^{-b}\,dx;\C^2),
\]
where the strong convergence follows from $q<3-b$ and Lemma~\ref{lem:weighted}.

For each fixed $R>0$,
\[
\sqrt{\mathfrak t_{\varepsilon_{a_n}}(\xi)}\rightarrow\frac{|\xi|}{\sqrt{2m}}\quad\text{uniformly on }B_R.
\]
Since $\widehat\varpi_{a_n}\rightharpoonup\widehat Q$
in $L^2(B_R;\C^2)$, we have
\[
\sqrt{\mathfrak t_{\varepsilon_{a_n}}}\,\widehat\varpi_{a_n}
\rightharpoonup \frac{|\xi|}{\sqrt{2m}}\widehat Q
\quad\text{in }L^2(B_R;\C^2).
\]
By weak lower semicontinuity,
\[
\liminf_{n\to\infty}\mathcal T_{\varepsilon_{a_n}}(\varpi_{a_n})\ge\frac1{2m}\int_{B_R}|\xi|^2\,|\widehat Q(\xi)|^2\,d\xi.
\]
Letting $R\to\infty$ and applying the monotone convergence theorem,  we conclude that
\begin{equation}\label{eq:pseudo-liminf}
\liminf_{n\to\infty}\mathcal T_{\varepsilon_{a_n}}(\varpi_{a_n})\ge\frac1{2m}\norm{\nabla Q}_2^2.
\end{equation}
In particular, $Q\in H^1(\R^3;\C^2)$.

Using weak lower semicontinuity, we also have
$\norm{Q}_2^2\le1$.
Combining \eqref{eq:rescaled-energy-identity},
\eqref{eq:pseudo-liminf}, and \eqref{eq:energy-limsup}
with the strong convergence $\varpi_{a_n}\to Q$ in
$L^q(\R^3,|x|^{-b}\,dx;\C^2)$, we obtain
\[
\mathscr E_q(Q)\le\liminf_{n\to\infty}\frac{e_{a_n}^\delta-\frac{ma_n}{2}}{a_n\varepsilon_{a_n}^2}
\le\limsup_{n\to\infty}\frac{e_{a_n}^\delta-\frac{ma_n}{2}}{a_n\varepsilon_{a_n}^2}\le e_q.
\]
Since $e_q<0$, we must have $Q\ne0$. If $0<\norm{Q}_2^2<1$, we apply the normalization argument in the proof of Lemma~\ref{lem:NLS-ground-set} to $Q$ and deduce that $\mathscr E_q(Q)>e_q$, contrary to $\mathscr E_q(Q)\le e_q$. Hence $\norm{Q}_2=1$. By the definition of $e_q$, we have $e_q\le\mathscr E_q(Q)\le e_q$. Thus $\mathscr E_q(Q)=e_q$ and $Q\in\mathcal G_q$. In particular,
\[
\frac{e_{a_n}^\delta-\frac{ma_n}{2}}{a_n\varepsilon_{a_n}^2}\rightarrow e_q,
\]
and, since $\varpi_{a_n}\rightharpoonup Q$ in $L^2(\R^3;\C^2)$ and $\norm{\varpi_{a_n}}_2=\norm{Q}_2=1$, we obtain
\[
\varpi_{a_n}\rightarrow Q\quad\text{in }L^2(\R^3;\C^2).
\]

To prove that $\varpi_{a_n}\to Q$ strongly in $H^{\frac12}(\R^3;\C^2)$, it remains to show that
\[
\int_{\R^3}|\xi|\,\bigl|\widehat{\varpi}_{a_n}(\xi)-\widehat Q(\xi)\bigr|^2\,d\xi\rightarrow0.
\]
For this purpose, we use the following identity and estimate:
\[
\mathfrak t_\varepsilon(\xi)=\frac{|\xi|^2}{\sqrt{m^2+\varepsilon^2|\xi|^2}+m},\quad
|\xi|\le\left(\varepsilon+\frac{2m}{R}\right)\mathfrak t_\varepsilon(\xi)\quad\text{for }|\xi|\ge R.
\]
Here the inequality follows from $\sqrt{m^2+\varepsilon^2|\xi|^2}\le m+\varepsilon|\xi|$.
Let $C_1=\sup\limits_n\mathcal T_{\varepsilon_{a_n}}(\varpi_{a_n})<\infty$.
Using the above bound for $|\xi|$ and $\varepsilon_{a_n}\to0^+$,
we obtain
\[
\limsup_{n\to\infty}\int_{|\xi|\ge R}|\xi|\,|\widehat\varpi_{a_n}(\xi)|^2\,d\xi\le \frac{2mC_1}{R}.
\]
For fixed $R>0$,  since $\varpi_{a_n}\to Q$ in
$L^2(\R^3;\C^2)$, we have
\[
\int_{B_R}|\xi|\,|\widehat\varpi_{a_n}-\widehat Q|^2\,d\xi\le R\norm{\varpi_{a_n}-Q}_2^2\rightarrow0.
\]
On the other hand,
\[
\int_{|\xi|\ge R}|\xi|\,|\widehat\varpi_{a_n}-\widehat Q|^2\,d\xi\le 2\int_{|\xi|\ge R}|\xi|\,|\widehat\varpi_{a_n}|^2\,d\xi
+2\int_{|\xi|\ge R}|\xi|\,|\widehat Q|^2\,d\xi.
\]
Combining the estimates on $B_R$ and on $\{\xi\in\R^3:|\xi|\ge R\}$,
we first take the upper limit as $n\to\infty$
and then let $R\to\infty$.
Since $Q\in H^1(\R^3;\C^2)$, we conclude that
\[
\int_{\R^3}|\xi|\,|\widehat\varpi_{a_n}-\widehat Q|^2\,d\xi\rightarrow0.
\]
Together with $\varpi_{a_n}\to Q$ in
$L^2(\R^3;\C^2)$, this proves that
\[
\varpi_{a_n}\rightarrow Q\quad\text{in }H^{\frac12}(\R^3;\C^2).
\]
Since the sequence $a_n\to0^+$ was arbitrary, both limits in
\eqref{eq:rescaled-energy-limit} follow.

Finally, assume that $a_n\to0^+$ and $\varpi_{a_n}\to Q$ strongly in
$H^{\frac12}(\R^3;\C^2)$ for some $Q\in\mathcal G_q$.
By Lemma~\ref{lem:weighted}, we have
$\mathcal A_q(\varpi_{a_n})\to\mathcal A_q(Q)$.
Combining the first limit in \eqref{eq:rescaled-energy-limit}
with \eqref{eq:rescaled-energy-identity}, we obtain
\[
\frac12\mathcal T_{\varepsilon_{a_n}}(\varpi_{a_n})\rightarrow e_q+\frac1q\mathcal A_q(Q)=\frac1{4m}\norm{\nabla Q}_2^2.
\]
This proves \eqref{eq:pseudo-kinetic-limit}.
\end{proof}

\begin{proof}[Completion of the proof of Theorem~\ref{thm:small-mass}]
The stated properties of $e_q$ and $\mathcal G_q$, including the limiting equation, follow from Lemma~\ref{lem:NLS-ground-set}. Using $I(u_a)=e_a^\delta$ and \eqref{eq:rescaled-energy-limit}, we obtain the first conclusion in \eqref{eq:main-energy-asymptotic}.

We now prove \eqref{eq:main-profile-convergence}. By \eqref{eq:limit-FW-coordinate}, \eqref{eq:normalized-dilation},
and \eqref{eq:varpi-rescaling}, we have
\begin{equation}\label{eq:rescaled-v-Fourier}
\widehat{\mathscr D_av_a}(\xi)=U(\varepsilon_a\xi)^{-1}\binom{\widehat\varpi_a(\xi)}{0}.
\end{equation}
We first show that
\[
\operatorname{dist}_{H^{\frac12}(\R^3;\C^4)}\left(\mathscr D_av_a,\mathcal G_q^{\mathrm{up}}\right)\rightarrow0.
\]
Fix $Q\in\mathcal G_q$. The matrices $U(\varepsilon_a\xi)^{-1}$ are unitary and converge pointwise to $I_4$ as $a\to0^+$. Applying the dominated convergence theorem, we obtain
\[
\mathcal F^{-1}\left[U(\varepsilon_a\xi)^{-1}\binom{\widehat Q(\xi)}{0}\right]\rightarrow\binom{Q}{0}\quad\text{in }H^{\frac12}(\R^3;\C^4).
\]
For every $f\in H^{\frac12}(\R^3;\C^4)$, it follows from the unitarity of $U(\varepsilon_a\xi)^{-1}$ that
\[
\left\|\mathcal F^{-1}\left[\bigl(U(\varepsilon_a\xi)^{-1}-I_4\bigr)\widehat f\right]\right\|_{H^{\frac12}}\le 2\norm{f}_{H^{\frac12}}.
\]
Using this bound and the compactness of $\mathcal G_q$ in $H^{\frac12}(\R^3;\C^2)$, we deduce by a finite covering argument that the convergence is uniform for $Q\in\mathcal G_q$. Therefore,
\[
\operatorname{dist}_{H^{\frac12}}\left(\mathscr D_av_a,\mathcal G_q^{\mathrm{up}}\right)
\le\operatorname{dist}_{H^{\frac12}}(\varpi_a,\mathcal G_q)+\sup_{Q\in\mathcal G_q}
\left\|\mathcal F^{-1}\left[\bigl(U(\varepsilon_a\xi)^{-1}-I_4\bigr)\binom{\widehat Q(\xi)}{0}\right]\right\|_{H^{\frac12}}.
\]
It follows from \eqref{eq:rescaled-energy-limit} and the uniform
convergence for $Q\in\mathcal G_q$ that both terms tend to zero.

By scaling and \eqref{eq:sharp-w-bounds},
\[
\norm{\mathscr D_aw_a}_2^2=\frac{\norm{w_a}_2^2}{a}=o(\varepsilon_a^2),
\]
\[
\int_{\R^3}|\xi||\widehat{\mathscr D_aw_a}(\xi)|^2\,d\xi=\frac1{a\varepsilon_a}
\int_{\R^3}|\xi||\widehat w_a(\xi)|^2\,d\xi\le\frac{\norm{w_a}_E^2}{a\varepsilon_a}=o(1).
\]
Thus $\mathscr D_aw_a\to0$ in $H^{\frac12}(\R^3;\C^4)$. By \eqref{eq:rescaled-v-Fourier} and unitarity,
$\norm{\mathscr D_av_a}_{H^{\frac12}}=\norm{\varpi_a}_{H^{\frac12}}=O(1)$. Since
\[
\mathscr D_au_a=t_a\mathscr D_av_a+\mathscr D_aw_a
\]
and $t_a\to1$, we have
\[
\operatorname{dist}_{H^{\frac12}}(\mathscr D_au_a,\mathcal G_q^{\mathrm{up}})\le \operatorname{dist}_{H^{\frac12}}
(\mathscr D_av_a,\mathcal G_q^{\mathrm{up}})+|t_a-1|\norm{\mathscr D_av_a}_{H^{\frac12}}
+\norm{\mathscr D_aw_a}_{H^{\frac12}}\rightarrow0.
\]
Since $\widetilde u_a=\mathscr D_au_a$, this proves \eqref{eq:main-profile-convergence}.

Let $a_n\to0^+$ be arbitrary. By Lemma~\ref{lem:rescaled-compactness}, after passing to a subsequence, there exists
$Q\in\mathcal G_q$ such that
\[
\varpi_{a_n}\rightarrow Q\quad\text{in }H^{\frac12}(\R^3;\C^2).
\]
For this $Q$, it follows from \eqref{eq:rescaled-v-Fourier}, \eqref{eq:sharp-w-bounds}, and the dominated convergence theorem that
\[
\widetilde u_{a_n}\rightarrow\binom{Q}{0}\quad\text{in }H^{\frac12}(\R^3;\C^4),
\]
which proves the subsequence statement in the theorem.

We next prove \eqref{eq:main-lower-spinor-limit}. Set
\[
\mathfrak M_a(\xi)=\frac{\rho_-(\varepsilon_a\xi)}{\varepsilon_a}S(\xi).
\]
Since $S(\varepsilon_a\xi)=S(\xi)$ and
\[
\lambda(\varepsilon_a\xi)-m=\frac{\varepsilon_a^2|\xi|^2}{\lambda(\varepsilon_a\xi)+m},
\]
we can write
\begin{equation}\label{eq:lower-multiplier-formula}
\mathfrak M_a(\xi)=\frac{\boldsymbol\sigma\cdot\xi}{\sqrt{2\lambda(\varepsilon_a\xi)
\bigl(\lambda(\varepsilon_a\xi)+m\bigr)}}.
\end{equation}
The lower component of $\varepsilon_a^{-1}\mathscr D_av_a$ has Fourier transform $\mathfrak M_a(\xi)\widehat\varpi_a(\xi)$. By $(\boldsymbol\sigma\cdot\xi)^*(\boldsymbol\sigma\cdot\xi)=|\xi|^2I_2$, we have
\[
\mathfrak M_a(\xi)^*\mathfrak M_a(\xi)=\frac{|\xi|^2}{2\lambda(\varepsilon_a\xi)
\bigl(\lambda(\varepsilon_a\xi)+m\bigr)}I_2=\frac{\mathfrak t_{\varepsilon_a}(\xi)}{2\lambda(\varepsilon_a\xi)}I_2.
\]
Consequently,
\begin{equation}\label{eq:lower-multiplier-norm}
\norm{\mathfrak M_a\widehat\varpi_a}_2^2
=\int_{\R^3}\frac{\mathfrak t_{\varepsilon_a}(\xi)}{2\lambda(\varepsilon_a\xi)}|\widehat\varpi_a(\xi)|^2\,d\xi.
\end{equation}
We first establish the limit of
$\norm{\mathfrak M_{a_n}\widehat\varpi_{a_n}}_2^2$.
For every fixed $R>0$,
\[
\frac{\mathfrak t_{\varepsilon_{a_n}}(\xi)}
{2\lambda(\varepsilon_{a_n}\xi)}
\rightarrow\frac{|\xi|^2}{4m^2}
\]
uniformly on $B_R$.
Combining this uniform convergence with
$\widehat\varpi_{a_n}\to\widehat Q$ in
$L^2(\R^3;\C^2)$, we deduce that
\[
\int_{B_R}
\frac{\mathfrak t_{\varepsilon_{a_n}}(\xi)}
{2\lambda(\varepsilon_{a_n}\xi)}
|\widehat\varpi_{a_n}(\xi)|^2\,d\xi
\rightarrow
\frac1{4m^2}\int_{B_R}
|\xi|^2|\widehat Q(\xi)|^2\,d\xi.
\]
Applying the same local argument to $\mathfrak t_{\varepsilon_{a_n}}$ on $B_R$ and subtracting the resulting limit from \eqref{eq:pseudo-kinetic-limit}, we obtain
\[
\int_{|\xi|>R}\mathfrak t_{\varepsilon_{a_n}}(\xi)|\widehat\varpi_{a_n}(\xi)|^2\,d\xi
\rightarrow\frac1{2m}\int_{|\xi|>R}|\xi|^2|\widehat Q(\xi)|^2\,d\xi.
\]
Since $\lambda(\varepsilon_{a_n}\xi)\ge m$, we have
\[
\limsup_{n\to\infty}\int_{|\xi|>R}\frac{\mathfrak t_{\varepsilon_{a_n}}(\xi)}{2\lambda(\varepsilon_{a_n}\xi)}
|\widehat\varpi_{a_n}(\xi)|^2\,d\xi\le\frac1{4m^2}\int_{|\xi|>R}|\xi|^2|\widehat Q(\xi)|^2\,d\xi\rightarrow0
\]
as $R\to\infty$. From the limit on $B_R$ and the estimate on $\{\xi\in\R^3:|\xi|>R\}$, we conclude that
\begin{equation}\label{eq:lower-multiplier-norm-limit}
\norm{\mathfrak M_{a_n}\widehat\varpi_{a_n}}_2^2\rightarrow\frac1{4m^2}\norm{\nabla Q}_2^2.
\end{equation}

Let $\Psi\in L^2(\R^3;\C^2)$ be compactly supported.
By \eqref{eq:lower-multiplier-formula}, we have
\[
\mathfrak M_{a_n}(\xi)
\longrightarrow\frac{\boldsymbol\sigma\cdot\xi}{2m}
\]
uniformly on $\operatorname{supp}\Psi$. Therefore
\[
\mathfrak M_{a_n}^*\Psi\rightarrow \frac{\boldsymbol\sigma\cdot\xi}{2m}\Psi
\quad\text{in }L^2(\R^3;\C^2).
\]
Using $\widehat\varpi_{a_n}\to\widehat Q$ in $L^2(\R^3;\C^2)$,
we obtain
\[
\ip{\mathfrak M_{a_n}\widehat\varpi_{a_n}}{\Psi}_2
=\ip{\widehat\varpi_{a_n}}{\mathfrak M_{a_n}^*\Psi}_2
\rightarrow\ip{\dfrac{\boldsymbol\sigma\cdot\xi}{2m}\widehat Q}{\Psi}_2.
\]

By \eqref{eq:lower-multiplier-norm-limit},
the sequence $\{\mathfrak M_{a_n}\widehat{\varpi}_{a_n}\}$
is bounded in $L^2(\R^3;\C^2)$.
Since compactly supported functions are dense in this space,
we deduce that
\[
\mathfrak M_{a_n}\widehat{\varpi}_{a_n}
\rightharpoonup
\frac{\boldsymbol\sigma\cdot\xi}{2m}\widehat Q
\quad\text{in }L^2(\R^3;\C^2).
\]
Moreover, using $(\boldsymbol\sigma\cdot\xi)^2=|\xi|^2I_2$ and Plancherel's theorem, we have
\[
\norm{\frac{\boldsymbol\sigma\cdot\xi}{2m}\widehat Q}_2^2
=
\frac{1}{4m^2}\norm{\nabla Q}_2^2.
\]
It follows from \eqref{eq:lower-multiplier-norm-limit} and the weak convergence in $L^2(\R^3;\C^2)$ that
\[
\mathfrak M_{a_n}\widehat{\varpi}_{a_n}
\rightarrow
\frac{\boldsymbol\sigma\cdot\xi}{2m}\widehat Q
\quad\text{strongly in }L^2(\R^3;\C^2).
\]
Write the rescaled positive spectral component in the standard $2+2$ form as
\[
\mathscr D_{a_n}v_{a_n}=\begin{pmatrix}
\varphi_{a_n}^{(v)}\\
\chi_{a_n}^{(v)}
\end{pmatrix},\quad
\varphi_{a_n}^{(v)},\chi_{a_n}^{(v)}\in L^2(\R^3;\C^2).
\]
Since the Fourier transform of
$\varepsilon_{a_n}^{-1}\chi_{a_n}^{(v)}$
is $\mathfrak M_{a_n}\widehat{\varpi}_{a_n}$,
we obtain by Plancherel's theorem
\begin{equation}\label{eq:lower-positive-limit}
\frac{\chi_{a_n}^{(v)}}{\varepsilon_{a_n}}\rightarrow-\frac{i}{2m}\boldsymbol\sigma\cdot\nabla Q
\quad\text{strongly in }L^2(\R^3;\C^2).
\end{equation}
Using $u_a-v_a=(t_a-1)v_a+w_a$,
$\norm{\mathscr D_av_a}_2=1$, and \eqref{eq:sharp-w-bounds},
we have
\[
\frac1{\varepsilon_a}\norm{\mathscr D_a(u_a-v_a)}_2\le\frac{|t_a-1|}{\varepsilon_a}
+\frac{\norm{w_a}_2}{\sqrt a\,\varepsilon_a}\rightarrow0.
\]
Together with \eqref{eq:lower-positive-limit}, this proves
\eqref{eq:main-lower-spinor-limit}.

By Lemma~\ref{lem:original-regular-branch}, the nonlinear term in the original equation belongs to $L^2(\R^3;\C^4)$. To determine the asymptotic behavior of $\mu_a$, we test the equation against $u_a$ and use \eqref{eq:limit-decomposition} to obtain
\begin{equation}\label{eq:multiplier-asymptotic-identity}
\mu_a a=t_a^2\norm{v_a}_E^2-\norm{w_a}_E^2-\mathcal A_q(u_a)+\mathcal A_p(u_a).
\end{equation}
In view of \eqref{eq:main-profile-convergence}, choose $Q_a\in\mathcal G_q$ such that
\[
\widetilde u_a-\begin{pmatrix}
Q_a\\
0
\end{pmatrix}\rightarrow0\quad\text{in }H^{\frac12}(\R^3;\C^4).
\]
From \eqref{eq:main-rescaling} and \eqref{eq:epsilon-balance}, we have 
$\frac{\mathcal A_q(u_a)}{a\varepsilon_a^2}=\mathcal A_q(\widetilde u_a)$. It follows from Lemma~\ref{lem:weighted}, \eqref{eq:NLS-Pohozaev}, and $\mathscr E_q(Q_a)=e_q$ that
\[
\mathcal A_q(\widetilde u_a)=\mathcal A_q(Q_a)+o(1)
=-\frac{2q}{2-\gamma_q}e_q+o(1).
\]
Using \eqref{eq:Ap-by-Aq} and Theorem~\ref{thm:main}, we have
\[
0\le\frac{\mathcal A_p(u_a)}{a\varepsilon_a^2}\le\norm{u_a}_\infty^{p-q}\frac{\mathcal A_q(u_a)}{a\varepsilon_a^2}
\le C_p a^{\frac{p-q}{2}}\frac{\mathcal A_q(u_a)}{a\varepsilon_a^2}\rightarrow0.
\]
Along the subsequence for which $\varpi_{a_n}\to Q$
in $H^{\frac12}(\R^3;\C^2)$, we obtain from
\eqref{eq:pseudo-kinetic-limit} that
\[
\frac{K_{a_n}}{a_n\varepsilon_{a_n}^2}=\mathcal T_{\varepsilon_{a_n}}(\varpi_{a_n})\rightarrow\frac1{2m}\norm{\nabla Q}_2^2.
\]
Moreover, $\norm{v_a}_E^2=ma+K_a=O(a)$. By \eqref{eq:sharp-w-bounds}, we have $t_a^2-1=o(\varepsilon_a^2)$ and $\norm{w_a}_E^2=o(a\varepsilon_a^2)$. Hence
\[
t_a^2\norm{v_a}_E^2-\norm{w_a}_E^2=\norm{v_a}_E^2+(t_a^2-1)\norm{v_a}_E^2-\norm{w_a}_E^2=ma+K_a+o(a\varepsilon_a^2).
\]
From \eqref{eq:multiplier-asymptotic-identity}, we then deduce
\[
\frac{m-\mu_{a_n}}{\varepsilon_{a_n}^2}
=\frac{\mathcal A_q(u_{a_n})-\mathcal A_p(u_{a_n})-K_{a_n}}
{a_n\varepsilon_{a_n}^2}+o(1)\rightarrow \mathcal A_q(Q)-\frac1{2m}\norm{\nabla Q}_2^2
=\Lambda_q,
\]
where the last equality follows by testing \eqref{eq:NLS-limit-introduction} against $Q$. Since every sequence $a_n\to0^+$ admits a subsequence along which $\frac{m-\mu_{a_n}}{\varepsilon_{a_n}^2}\to\Lambda_q$, the second conclusion in \eqref{eq:main-energy-asymptotic} follows.

Since $u_a^-=w_a$, we obtain from \eqref{eq:sharp-w-bounds} that
\[
\frac{\norm{u_a^-}_E}{\sqrt a\,\varepsilon_a}
\rightarrow0.
\]
Along any subsequence for which \eqref{eq:main-lower-spinor-limit} holds, we also have
\[
\frac{\|\chi_{a_n}\|_2^2}
     {a_n\varepsilon_{a_n}^2}
\rightarrow
\frac{1}{4m^2}\|\boldsymbol\sigma\cdot\nabla Q\|_2^2
=
\frac{1}{4m^2}\|\nabla Q\|_2^2
=
-\frac{\gamma_q}{m(2-\gamma_q)}e_q,
\]
where the last equality follows from \eqref{eq:NLS-Pohozaev} and $\mathscr E_q(Q)=e_q$. The limit is independent of $Q\in\mathcal G_q$, and
every sequence $a_n\to0^+$ admits such a subsequence. Therefore,
\[
\frac{\norm{\chi_a}_2^2}{a\varepsilon_a^2}\rightarrow-\frac{\gamma_q}{m(2-\gamma_q)}e_q\quad\text{as }a\to0.
\]

Finally, the first limit in \eqref{eq:main-unscaled-limit} follows directly from Theorem~\ref{thm:main}. Set $z_a=\frac{u_a}{\sqrt a}$. We prove that $z_a\rightharpoonup0$ in $H^1(\R^3;\C^4)$ as $a\to0^+$.
Since the truncation is inactive, we obtain from
\eqref{eq:nonlinear-L2-bound} that
\[
\left\||x|^{-b}\bigl(|u_a|^{q-2}-|u_a|^{p-2}\bigr)u_a\right\|_2\le C a^{\frac{q-1}{2}}.
\]
Consequently,
\[
\norm{(H_0-\mu_a)z_a}_2\le C a^{\frac{q-2}{2}}\rightarrow0.
\]
Thus $H_0z_a=\mu_a z_a+o(1)$ in $L^2(\R^3;\C^4)$. Since $\norm{z_a}_2=1$,
\[
\left|\norm{H_0z_a}_2-\mu_a\right|\le\norm{(H_0-\mu_a)z_a}_2\rightarrow0.
\]
Using $\mu_a\to m$ and the identity
\[
\norm{H_0z_a}_2^2=\norm{\nabla z_a}_2^2+m^2\norm{z_a}_2^2,
\]
we deduce that $\norm{\nabla z_a}_2\to0^+$.

By Theorem~\ref{thm:main}, the family $\{z_a\}$
is bounded in $H^1(\R^3;\C^4)$.
If $a_n\to0^+$ and $z_{a_n}\rightharpoonup z$ in
$H^1(\R^3;\C^4)$, then $\nabla z=0$.
Thus $z$ is constant, and since $z\in L^2(\R^3;\C^4)$,
we have $z=0$.
Therefore $z_a\rightharpoonup0$ in $H^1(\R^3;\C^4)$
as $a\to0^+$.
This completes the proof.
\end{proof}

\section*{Statements and Declarations}

\subsection*{Funding}
This work is partially supported by the NSFC (12471102), NSF of Jilin Province (20250102004JC), and Research Project of the Education Department of Jilin Province (JJKH20250296KJ).

\subsection*{Competing interests}
The authors declare that they have no conflict of interest.

\subsection*{Data availability}
Data sharing is not applicable to this article as no datasets were generated or analysed during the current study.

\end{document}